\documentclass[12pt]{amsart}
\usepackage{amssymb}
\usepackage{amsmath}
\usepackage{amsfonts}
\usepackage{epsfig}
\usepackage{color}
\usepackage{epstopdf}
\usepackage{graphicx}
\usepackage{amsthm}
\usepackage{enumerate}
\usepackage[mathscr]{eucal}
\usepackage{verbatim}

 \theoremstyle{plain}
\newtheorem{theorem}{Theorem}[section]
\newtheorem{lemma}[theorem]{Lemma}

\newtheorem{proposition}[theorem]{Proposition}

\newtheorem{defn}[theorem]{Definition}
\newtheorem{rmk}[theorem]{Remark}

\theoremstyle{plain}

\theoremstyle{remark}

\newcommand{\Real}{\mathbb R}

\newcommand{\mbt}{\mathbf t}

\newcommand{\mba}{\mathbf a}
\newcommand{\mbb}{\mathbf b}

\newcommand{\cuxi}{\mathfrak I}

\newcommand{\chg}[1]{{\color{red}#1}}

\DeclareMathOperator{\dist}{dist}

\newcommand{\ef}[1]{f_{\cuxi^{#1}}}
\newcommand{\eff}[2]{f_{\cuxi^{#1}_{#2}}}
\newcommand{\cu}[1]{\cuxi^{#1}}
\newcommand{\cuu}[2]{\cuxi^{#1}_{#2}}
\newcommand{\tla}{T_\lambda^\gamma}

\newcommand{\mga}{M^\gamma_\tau}

\newcommand{\gatha}{\gamma_{\tau}^{h,\mba}(t)}
\newcommand{\ma}[1]{M_{#1}^{\gamma,\mba}}
\newcommand{\wei}{w_\gamma^\alpha}
\newcommand{\tlawf}{\tla [w,f]}
\newcommand{\tlaww}[1]{\tla [#1,f]}

\newcommand{\gao}{\gamma_{\circ}}
\newcommand{\gth}{\gamma_\tau^h}
\newcommand{\clag}{\mathfrak G(\epsilon)}

\newcommand{\dsiga}{d\sigma_{\!A,h}^\mba}
\newcommand{\siga}{\sigma_{\!A,h}^\mba}
\newcommand{\gathaa}{\gamma_\tau^{h,\mba}}

\title[Restriction estimates for space curves]{Restriction Estimates for space curves with respect to general measures}

\author{Seheon Ham}
\author{Sanghyuk Lee}
\thanks{Supported in part by  NRF grant 2012008373 (Republic of Korea).}
\address{Department of Mathematical Sciences, Seoul National University, Seoul 151-747, Republic of Korea}
\email{hamsh@snu.ac.kr} \email{shklee@snu.ac.kr}

\begin{document}
\begin{abstract}
In this paper we consider adjoint restriction estimates for space
curves with respect to general measures and obtain  optimal
estimates when the curves satisfy a finite type condition. The
argument here is new in that it doesn't rely on  the \emph{offspring
curve} method,  which has been extensively used in the previous
works. Our work was inspired by the recent argument due to Bourgain
and Guth which was used to deduce linear restriction estimates from
multilinear estimates for hypersurfaces.
\end{abstract}

\maketitle

\section{introduction}

Let $\gamma:I=[0,1] \to \mathbb {R}^d,\, d\ge 2 $ be a smooth
function. For $\lambda\ge 1$ we define an oscillatory integral
operator by
\[ T_\lambda^\gamma f(x)= \int_I e^{i\lambda x\cdot \gamma(t)} f(t) dt. \]
This operator is an adjoint form of the Fourier restriction to the
curve $\lambda\gamma(t)$, $t\in I$. Let $\nu$ be a measure in
$\mathbb R^d$ and $1\le p,q\le \infty$. We consider  the oscillatory
estimate
\begin{equation}\label{exten}
\| T_\lambda^\gamma f\|_{L^q(d\nu)} \leq C\lambda^{-\beta}
\|f\|_{L^p(I)}.
\end{equation}

\subsection*{\it Nondegenerate curves} It is well known that the range of $p$,
$q$ is related to the curvature condition of $\gamma$. When $\nu$ is
the Lebesgue measure the problem of obtaining the estimate
\eqref{exten} has been considered by many authors \cite{Z,
presti,christ,drury} (also see \cite{DM1,DM2,BOS1,BOS2,BOS3,DeW}).
Under the assumption
\begin{equation}\label{torsion}
\det(\gamma'(t), \gamma''(t), \cdots, \gamma^{(d)}(t)) \neq 0
\end{equation}
for all $t\in I$,  which we call the nondegeneracy condition,
it is known that \eqref{exten} holds with $\beta=d/q$ if
\begin{equation}\label{opt}
 \frac{d(d+1)}{2q} + \frac{1}{p} \le 1 \quad \textrm{and}\quad q > \frac{d^2 + d +2}{2}\,.
\end{equation}
In two dimension this is due to Zygmund \cite{Z} and a
generalization to oscillatory integral was obtained by H\"ormander
\cite{H} (see \cite{F} for earlier work by Fefferman and Stein). In
higher dimensions $d\ge 3$  the estimates on the whole range were
proved by Drury \cite{drury} after earlier partial results due to
Prestini \cite{presti} and Christ \cite{christ}. Necessity of the
condition ${ d(d+1) / 2q} + {1 / p} \le 1$ can be shown by a Knapp
type example. When $\gamma(t) = (t, t^2, \cdots, t^d)$ and $d\ge 3$,
by a result due to Arkhipov, Chubarikov and Karatsuba \cite{act} it
follows that the condition $q>(d^2 + d +2)/2$ is necessary. The
operator $T_\lambda^\gamma$ can also be generalized by replacing
$x\cdot \gamma(t)$ with $\phi(x,t)$. In this case, Bak and the
second author \cite{bl} showed that \eqref{exten} holds with $\beta
= d/{q}$ for $p,q$ satisfying \eqref{opt} whenever $\det (
\partial_t (\nabla_x \phi), \partial_t^2 (\nabla_x
\phi),\cdots,\partial_t^d (\nabla_x \phi)) \neq 0$ holds. Bak,
Oberlin and Seeger \cite{BOS1} showed a weak type estimate for the
critical $p=q={(d^2 + d +2)/ 2}$.

In this paper, we are concerned with  $L^p$--$L^q$ estimate for
$T_\lambda$ with respect to general measures other than the Lebesgue
measure. More precisely, for $0 < \alpha \leq d$, let $\mu$ be a
positive Borel measure which satisfies
\begin{equation}\label{bmeasure}
\mu ( B(x, \rho) ) \leq C_\mu \rho^\alpha, \quad \rho >0
\end{equation}
for any $x\in \Real^d$. Here $C_\mu$ is independent of $x$, $\rho$.
Considering $f=\chi_{[0,1]}$, one easily sees that the best possible
$\beta$ for \eqref{exten} is $\alpha/q$ when $\nu\, (=\mu)$
satisfies \eqref{bmeasure}. In fact, note that $|T_\lambda
f(x)|\gtrsim 1$ if $|x|\le c\lambda^{-1}$ for a sufficiently small
$c>0$. We aim to find  the optimal range of $(p,q)$ for which the
inequality
\begin{equation}\label{frac} \| T_\lambda^\gamma f\|_{L^q(d\mu)}
\le C\lambda^{-\frac \alpha q} \|f\|_{L^p(I)}
\end{equation}
holds under the assumption that $\mu$ satisfies \eqref{bmeasure}.

In order to state our results we define a number
$\beta=\beta(\alpha)$ by setting
\[
\beta(\alpha)=(j+1) \alpha+\frac{(d-j-1)(d-j)}{2}
\]
if $d-j-1<\alpha\le d-j$ for  $j=0,\dots, d-1$. Note that
$\beta(\alpha)$ continuously increases as $\alpha$ increases.
\begin{figure}[t] \label{fig1}
\centerline{\epsfig{file = 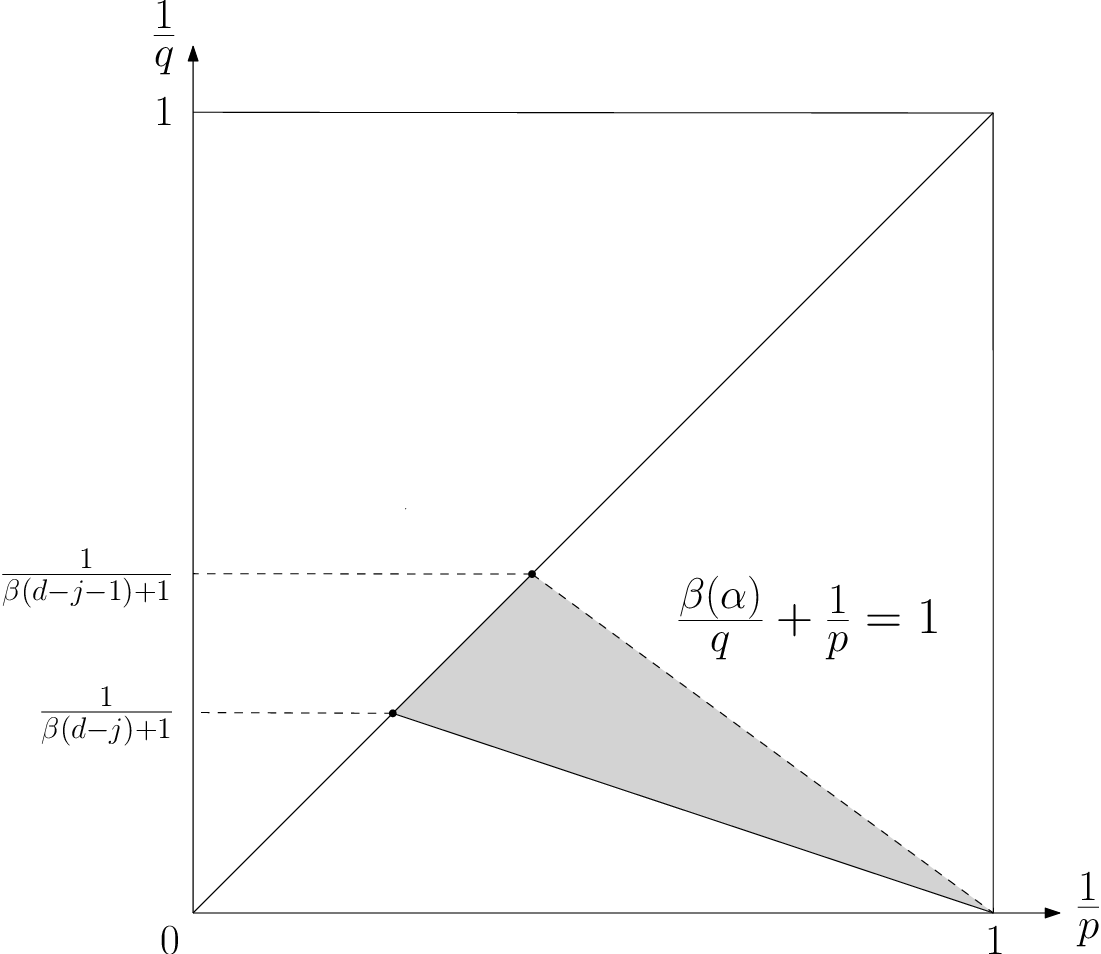, width=0.45\textwidth}}
\caption{\small For $d-j-1 < \alpha\leq d-j$, the edge line
${\beta(\alpha)}/{q} + {1}/{p} =1$ is contained in the shaded area.
The region given by ${\beta(\alpha)}/{q} + {1}/{p} < 1$ gets larger
as $\alpha$ decreases.}
\end{figure}
The following is our first result.

\begin{theorem}\label{mainthm}
Let $\gamma\in C^{d+1}(I)$ and $0<\alpha\le d$ and let $\mu$ be a
positive  Borel measure. Suppose that $\gamma$ and $\mu$ satisfy
\eqref{torsion} and \eqref{bmeasure}, respectively. Then, for $1\le
p,q\le \infty$ satisfying $d/{q}\le (1-1/p)$, $q\ge 2d$, and
\[
\beta(\alpha)/q+1/p<1, \,\, q
> \beta(\alpha)+1,
\]
 there exists a constant $C$ such that \eqref{frac} holds for $f
\in L^p(I)$ and $\lambda\ge 1$.
\end{theorem}

As $\alpha$ decreases the admissible range of $p,q$ gets larger (see
Figure \ref{fig1}). If $\alpha =d$, this extends Drury's result
\cite{drury} to general measures except for the end line case
$\beta(d)/q+1/p=1$. (Note that $\beta(d)=(d^2+d)/2$.) The condition
$d/{q}\le (1-1/p)$, $q\ge 2d$ is related to application of
Plancherel's theorem which gives $d$-linear estimates (see Lemma
\ref{l2}). The condition $\beta(\alpha)/q+1/p<1$ is sharp in that
there is a measure satisfying \eqref{bmeasure} but \eqref{frac}
fails if $\beta(\alpha)/q+1/p>1$ (see Appendix A). The restriction
$q> \beta(\alpha)+1$ also seems necessary in general even though at
present we  know it only in special cases. Note that $\beta(\alpha)>
d$ if $\alpha> 1$ and $\beta(\alpha)+1> 2d$ if $\alpha> 2$. Hence,
the assumption $d/{q}\le (1-1/p)$, $q\ge 2d$ is redundant when
$\alpha> 2$. In particular, when $\mu$ is the surface measure on a
compact smooth hypersurface $\Sigma\subset \mathbb R^d$ and $d\ge
3$, by rescaling the estimate \eqref{frac}  we get
\[\| T_1 f\|_{L^q(\lambda \Sigma)} \le C\|f\|_{L^p(I)}\]
provided that $(d+2)(d-1)/(2q)+1/p<1$ and $p\le q$. This can be seen
as a generalization of  $ L^p(S^1)$-$L^q(\lambda S^1)$ bound
\cite{bcsv} (also see \cite{grs} and \cite{bs} for related results)
for the extension operator from the circle $S^1$ in $\mathbb R^2$ to
the large circle $\lambda S^1$.

Our results here rely on the so-called \emph{multilinear approach}
which has been used to study the restriction problem for
hypersurfaces (\textit{cf.} \cite{becata, tavave}). Especially we
adapt the recent argument due to Bourgain and Guth \cite{bg} (also
see references therein) which was successful in deducing linear
estimate from multilinear one. For the space curves with non
vanishing torsion the sharp $d-$linear (extension) estimate is a
straightforward consequence of Plancherel's theorem under the
assumption that the support functions are separated (see Lemma
\ref{l2} and Lemma \ref{l2fractal}). Then it is crucial to control
$T_\lambda f$ by products of $T_\lambda f_1, \dots, T_\lambda f_d$
for which the supports of $f_i$ are separated from one another while
the remaining parts are bounded by a sum of  $|T_\lambda g|$ with
$g$ supported in a small interval. (See Lemma \ref{multidecomp}.)
Compared with \cite{bg} this is relatively simpler since we only
have to deal with one parameter separation in order to make use of
the multilinear estimate. To close induction we need to obtain
uniform estimates which do not depend on  particular choices of
curves. After proper normalization we can reduce the matter to
dealing with a class of curves which are close to a monomial curve.
An obvious byproduct of this approach is stability of estimates over
a family of curves (see Remark \ref{uniform}).

The estimates of the endpoint line case ($\beta(\alpha)/q +1/p =1$)
are not likely to be possible  with  general measures satisfying
\eqref{bmeasure}. 
But they still look plausible with specific measures which satisfy
certain regularity assumptions. However, these endpoint estimates
are  beyond the method of this paper. On the other hand, one may try
to use the method based on offspring curves \cite{drury} but a
routine adaptation of the presently known argument only gives
\eqref{frac} on a smaller range, namely $d/{q}\le (1-1/p)$,
$q\ge 2d$, $\beta(d)/q +1/p <1$ and $q > \beta(d) +1$.
%
%
\subsection*{\it Finite type curves}
There are also results when curves degenerate, namely the condition
\eqref{torsion} fails. Let us set  $\mba = (a_1,\cdots,a_d)$ with
positive integers $a_1, a_2, \cdots, a_d$ satisfying $a_1<a_2<
\cdots <a_d$. Then for $t\in I$ we also set
\begin{equation}\label{defm} \ma{t}= \begin{bmatrix} \gamma^{(a_1)}(t),& \gamma^{(a_2)}(t),
&\cdots,& \gamma^{(a_d)}(t) \end{bmatrix},
\end{equation}
where the column vectors $\gamma^{(a_i)}(t)$ are $a_i$--th
derivatives of $\gamma$. So, $\gamma$ is nondegenerate at $t$ if
$\det \ma{t} \neq 0$ with $\mba = (1,2,\cdots, d)$. We recall the
following definition which was introduced in \cite{christ}.

\begin{defn} Let $\gamma:I=[0,1] \to \mathbb {R}^d,\, d\ge 2 $ be a
smooth curve.
 We say that $\gamma$ is of finite
type at $t\in I$ if there exists $\mba = (a_1,\cdots,a_d)$ such that
$ \det\ma{t} \neq 0$. We also say that $\gamma$ is of finite type if
so is $\gamma$ at every $t\in I$.
\end{defn}

When degeneracy appears the boundedness of $T_\lambda^\gamma$ is no
longer the same so that \eqref{frac} holds only on a smaller set of
$p,q$.  When $\mu$ is the Lebesgue measure Christ \cite{christ}
obtained some sharp restriction estimates for the curves of finite
type on a restricted range. On the other hand, a natural attempt is
to recover the full range \eqref{opt} by introducing a weight which
mitigates bad behavior at degeneracy. In fact, let us consider the
estimate
\[ \| \tlaww {w} \|_{L^q(d\mu)} \leq C \lambda^{-\frac \alpha q}
\| f \|_{L^p (w dt)},\] where $\lambda\ge 1$ and
\begin{equation*}
\tlawf (x) = \int_I e^{i \lambda x \cdot \gamma(t) } f(t) w(t) dt.
\end{equation*}
The dual form of this estimate with $\lambda=1$ is
\begin{equation} \label{lafree}
\Big(\int_I
 |\widehat { g d\mu}(\gamma (t))|^{p'} w(t) dt\Big)^{1/p'} \le
C  \|g\|_{L^{q'}(d\mu)}.
\end{equation}

 There has been a long line of investigations on the estimate \eqref{lafree}
 \cite{DM1,DM2, D2, BO, BOS1,BOS2,BOS3,DeW,DeM}  when  $\mu$ is the Lebesgue measure and  $w dt$ is the affine
arclength measure.  When $d=2$, it was shown by Sj\"olin \cite{sj}
(also see \cite{ob}). In higher dimensions the study on
\eqref{lafree} was carried out by Drury and Marshall \cite{DM1},
\cite{DM2}. Drury \cite{D2}, Bak and Oberlin \cite{BO} obtained
partial results for specific classes of curves in $\mathbb R^3$. If
$I=\mathbb R$, by scaling the condition ${d(d+1)}/(2p')=1/q$ is
necessary for \eqref{lafree}. Wright and Dendrinos \cite{DeW}
obtained a uniform estimate for a class of polynomial curves on the
range $(d^2 + 2d )/2< q\le \infty$. This result was extended to a
larger region \cite{BOS3} (see Section 8). There is also a result
for the curves of which components are rational functions rather
than polynomials (see \cite{DeFW}). Bak, Oberlin and Seeger obtained
the estimates on the full range including the weak endpoint estimate
for the monomial curves and the curves of simple type \cite{BOS3}.
Dendrinos and M\"uller \cite{DeM} further extended this result to
the curves of small perturbation of monomial curves and for the
critical case $p=q={(d^2 + d +2)/ 2}$ the weak type endpoint
estimate also holds for these curves (see \emph{Remark} in Section 6
of \cite{BOS3}). The problem of obtaining \eqref{lafree} is now
settled for the finite type curves which are defined locally though
the uniform estimate is still open when curves are given on the
whole real line.

\ In what follows  we consider  $L^p$--$L^q$ estimate of $\tlaww
{\wei}$ with respect to the measure $\mu$ satisfying
\eqref{bmeasure}. Let us define a measure by setting
\begin{equation*}
w_\gamma^\alpha(t)dt  = | \det(\gamma'(t) , \gamma''(t), \cdots,
\gamma^{(d)}(t)) |^{\frac1{\beta(\alpha)}} dt.
\end{equation*}
When $\alpha=d$ this coincides with the affine arclength measure on
$\gamma$. Considering a monomial curve and a measure satisfying
 the homogeneity condition $\int g(\lambda x)
d\mu(x)=\lambda^{-\alpha}\int g(x) d\mu(x)$ (for example the measure
$\mu$ given in Appendix A), by rescaling  one can easily see that
the exponent $1/\beta(\alpha)$ is the correct choice in order that
the estimate \eqref{tlaw} holds for $p,q$ satisfying
$\beta(\alpha)/q + 1/p \le 1$. In \cite{BOS3} (see Section 2),  when
$\mu$ is the Lebesgue measure it was shown that  the optimal power
of torsion  is $1/\beta(d)=2/d(d+1)$ so that \eqref{lafree} holds
for $d(d+1)/(2q) + 1/p \le 1$. If we consider the induced Lebesgue
measure on  lower dimensional hyperplanes, this clearly shows that
our choice of $\beta(\alpha)$ is optimal at least if $\alpha$ is an
integer.

Our second result  reads as follows.

\begin{theorem}\label{finitethm}
Let $\gamma\in C^{\infty}(I)$ and $0<\alpha\le d$. Suppose that
$\mu$ satisfies \eqref{bmeasure} and $\gamma$ is of finite type.
Then, for $1\le p,q\le \infty$ satisfying $d/{q}\le (1-1/p)$, $q\ge
2d$ and $\beta(\alpha)/q + 1/p < 1$, $q
> \beta(\alpha) + 1$, there exists a constant $C$ such that
\begin{equation}\label{tlaw}
 \| \tlaww {\wei} \|_{L^q(d\mu)} \leq C \lambda^{-\frac \alpha q} \| f \|_{L^p (\wei
 dt)}.
\end{equation}
\end{theorem}

This generalizes the previous results to general measures except for
$p,q$ which are on the end line. Thanks to the finite type
assumption a suitable normalization by a finite decomposition and
rescaling reduce the problem to the case of monomial type curves of
which degeneracy only appears a single point. Further decomposition
away from the degeneracy enables us to obtain the desired
estimate \eqref{tlaw} by relying on the stability of estimates
for non-degenerate curves.

\smallskip

The paper is organized as follows. In Section 2 we prove Theorem
\ref{mainthm}. In Section 3 we give the proof of Theorem
\ref{finitethm} which is based on Theorem \ref{mainthm}. Sharpness
of the condition $\beta(\alpha)/q + 1/p < 1$ will be shown in
Appendix A.

\section{Proof of Theorem \ref{mainthm} }

The proof of Theorem \ref{mainthm} is based on an adaptation of
Bourgain-Guth argument in \cite{bg}, which relies on a multilinear
estimate and  uniform control of estimates over classes  of curves
and measures. This requires proper normalization of them.

\subsection*{\it Normalization of curves} For $a,b\in \mathbb
R$, $a\neq b$, we set
\[[a,b]^*=\begin{cases} \,\, [a,b] \,\, \text{ if } a<b,\\ \,\, [b,a] \,\, \text{ if } b<a. \end{cases}\]
Let $\gamma\in C^{d+1}(I)$ satisfying \eqref{torsion}, and let $\tau
\in I$ and $h$ be a real number  such that $[\tau,\tau+h]^*\subset
I$. Then let us define a $d\times d$ matrix $M_\tau^\gamma$
and a diagonal matrix $ D_h$ by
\begin{align*}\label{transform}
M_\tau^\gamma& = M_\tau^{\gamma, (1,2,\dots, d)}= (
\gamma'(\tau),\gamma''(\tau),\cdots,\gamma^{(d)}(\tau) ),\\
\quad D_h&=(he_1, h^2e_2, \dots, h^{d} e_d).
\end{align*}
We also set
\begin{equation}
\label{normalcur} \gth(t)= D_h^{-1}(M_\tau^\gamma)^{-1
}(\gamma(ht+\tau)-\gamma(\tau)).
\end{equation}
Then it follows that
\begin{equation}\label{normalcurve}
x\cdot(\gamma(ht+\tau)-\gamma(\tau))=D_h (M_\tau^\gamma)^t x\cdot
\gamma_\tau^h(t).
\end{equation}

  Let us set
\[\gamma_\circ(t)=\Big(t,\frac{t^2}{2!},\dots, \frac{t^d}{d!}\,\Big). \]
For a given $\epsilon>0$ we define the class $\mathfrak G(\epsilon)$
of curves   by setting
\[
\mathfrak G(\epsilon)= \Big\{\gamma \in C^{d+1}(I) :
\|\gamma-\gao \|_{C^{d+1}(I)}\le \epsilon \Big\}.
\]

\begin{lemma} \label{normalization}
Let  $\gamma\in C^{d+1}(I)$ satisfying \eqref{torsion} and let
$\tau\in I$. Then, for $\epsilon>0$ there is a {constant} $\delta>0$
such that $\gth\in \mathfrak G(\epsilon)$ whenever
$[\tau,\tau+h]^*\subset I$ and $0<|h|\le \delta$.
\end{lemma}

For a given matrix $M$,  $\| M \| $ denotes the usual matrix norm
$\max_{|x|=1} | Mx| $.

\begin{proof} It is enough to consider the case $[\tau,\tau+h]\subset
I$. The other case $[h+\tau,\tau]\subset I$ can be shown similarly.
By Taylor's expansion
\begin{align*}
\gamma(ht+\tau)-\gamma(\tau)
&=\gamma^{'}\!(\tau)ht+\gamma^{''}\!(\tau)h^2\frac{t^2}{2!}+\dots+
\gamma^{(d)}\!(\tau)h^d\frac{t^d}{d!}+ \mathcal E(\tau,h,t)\\
&=M_\tau^\gamma D_h \gao(t)+\mathcal E(\tau,h,t)
\end{align*}
with $\|\mathcal E(\tau,h,t)\|_{C^{d+1}(I)}\le Ch^{d+1}$ uniformly
in $\tau$. Since $\gth(t)= \gao(t)+( M_\tau^\gamma D_h)^{-1}$
$\mathcal E(\tau,h,t),$
\begin{equation}\label{error}
\|\gth-\gao\|_{C^{d+1}(I)}\le C\|(M_\tau^\gamma)^{-1}\|h.
\end{equation}
By continuity $\|(M_\tau^\gamma)^{-1}\|$ is uniformly
 bounded along $\tau\in I$ by a constant $B$ because $\gamma$ satisfies
\eqref{torsion} and $\gamma\in C^{d+1}(I)$. Taking
$\delta=\epsilon/(2CB)$, we see $\gth\in \clag$.
\end{proof}

\begin{rmk}\label{interval} Let $J\subset I$. From the proof of Lemma \ref{normalization} it is
clear that if $\|\gamma\|_{C^{d+1}(J)}\le B_1$ and
$\|(M_\tau^\gamma)^{-1}\|\le B_2$ for all $\tau \in J$, then for
$\epsilon>0$ there is a $\delta=\delta(B_1,B_2)>0$ such that
$\gth\in \clag$ provided that $[\tau,\tau+h]^*\subset J$ and
$0<|h|\le \delta$.
\end{rmk}

\subsection*{\it Rescaling of measures} For $M>0$ we denote
by $\mathfrak{B} (\alpha, M)$ the set of positive Borel measures
which satisfy \eqref{bmeasure} with $C_\mu=M$. If $\sigma\in
\mathfrak{B} (\alpha, M)$, $\sigma$ is finite on all compact sets in
$\Real^d$ by $\eqref{bmeasure}$. Hence,  $\sigma$ is a Radon measure
because $\Real^d$ is locally compact Hausdorff space. (See Theorem
7.8 in \cite{Fol}.)

For $\mba = (a_1,\cdots,a_d)$ let us define
\begin{equation}\label{diagonal}
D_h^\mba=(h^{a_1}e_1, h^{a_2}e_2, \dots, h^{a_d}e_d).
\end{equation}
Let $\sigma\in \mathfrak{B} (\alpha, M)$, $0< |h| < 1$, and let $A$
be a  $d\times d$ nonsingular matrix. Then the map  $F\to \int F(
D_h^\mba A\, x) d \sigma (x)$ defines a positive  linear functional
on $C_c(\mathbb R^d)$. By the Riesz representation theorem there
exists a unique Radon measure $\siga$ such that
\begin{equation}\label{normalmeasure}
\int F(x) \dsiga(x) =\int F( D_h^\mba A\, x) d \sigma (x)
\end{equation}
for any compactly supported continuous function $F$.

\begin{lemma}\label{rescale} Let\, $\mba =
(a_1,\dots,a_d)$ and $a_1, a_2, \dots, a_d$ satisfy that $0<a_1<a_2<
\ldots <a_d$ and $a_i\ge i$. If $\sigma\in \mathcal B(\alpha, M)$
and $A$ is a nonsingular matrix, then $\siga$ is also a Borel
measure which satisfies
\begin{equation}\label{decom}
\siga (B (x,\rho)) \leq C M\|A^{-1} \|^\alpha
|h|^{\frac{d^2+d}{2}-\beta(\alpha)-\sum_{i=1}^d a_i} \rho^\alpha
\end{equation}
for $(x,\rho)\in \mathbb R^d\times\mathbb R_+$. Here $C$ is
independent of $h, A$.
\end{lemma}

\begin{proof}
Let $d-j-1<\alpha\le d-j$ for  some $j=0,\dots, d-1$. By translation
we may assume $x=0$. To show \eqref{decom}, we consider $\siga$ as a
measure which is given by composition of two transformations on
$\sigma$.

We first consider the  measure $\sigma_A$ defined by
\[
\int F(x) d\sigma_A = \int F(A x) d\sigma(x), \quad F\in C_c(\mathbb R^d) \] with a nonsingular
matrix $A$. Then it follows that
\[
\sigma_A(B(0,\rho)) \le  M \| \omega \|^\alpha \rho^\alpha,
\]
where $\|\omega\|= \max_k |\omega_k|$ and $\omega_1,\dots,\omega_d$
are the column vectors of $A^{-1}$. In fact, $A x \in B(0,\rho)$
implies that $x$ can be written as a linear combination of
$\omega_k$ with a coefficient vector in $B(0,\rho)$, i.e. $x  = z_1
w_1 + \cdots +z_d \omega_d$ with $(z_1,\dots,z_d)\in B(0, \rho)$.
Since $\sigma \in \mathfrak B(\alpha,M)$,  we have
\begin{equation}\label{sigma_A}
\sigma_A(B(0,\rho)) = \int \chi_{B(0,\rho)} (Ax ) d\sigma(x) \le C
\int \chi_{B(0, d\|\omega\|\rho)} (y ) d\sigma(y) \le C M (\|
\omega\|\rho)^\alpha.
\end{equation}

Let us define a measure $\sigma_h^\mba$  by
\[\int F(x) d
\sigma_h^\mba = \int F(D_h^\mba x) d\sigma,\]
and claim that
\begin{equation}\label{sigma_h^a}
\sigma_h^\mba (B (0,\rho))\le C M |h|^{\frac{d^2+d}{2}-\sum_{i=1}^d
a_i-\beta(\alpha)}\rho^\alpha.
\end{equation}
Since $\{x: Ax\in B(0,\rho)\}\subset B(0, d\|\omega\|\rho)$, by
\eqref{sigma_A} and \eqref{sigma_h^a} it follows that
\begin{align*}
\siga ( B(0,\rho)) &= \int \chi_{B(0,\rho)} (A x) d\sigma_h^\mba(x)
\le C \sigma_h^\mba( B(0, d\|\omega\|\rho)) \\
&\le C M  \|\omega\|^\alpha |h|^{\frac{d^2+d}{2}-\sum_{i=1}^d
a_i-\beta(\alpha)}\rho^\alpha,
\end{align*}
and therefore we get \eqref{decom}.

Now it remains to show \eqref{sigma_h^a}. Let us set $$S=\{y:
D_h^\mba y \in B(0,\rho)\}.$$ Then, if $y=(y_1,\dots, y_d)\in S$ we
have $| y_i | \le |h|^{- a_i}\rho$. Hence, $S$ is contained in  the
rectangle $\mathcal{R}$ of dimension $|h|^{-a_1} \rho \times
|h|^{-a_2}\rho \times \cdots \times  |h|^{-a_d} \rho$. So, $
\mathcal{R} $ can be covered by as many as $O(|h|^{-a_1+1} \times
|h|^{-a_2+2} \times \cdots \times |h|^{-a_d + d})$ rectangles
$\mathcal R'$ of which dimension is $|h|^{-1} \rho\times
 |h|^{-2} \rho\times \dots \times  |h|^{-d}$
while each $\mathcal R'$ is covered by $O(1\times\cdots\times1\times
|h|^{-1} \times \cdots \times |h|^{-(d-j-1)})$ cubes of sidelength $
|h|^{-j-1} \rho$. Hence $\mathcal R$ is covered by cubes
$\mathcal{B}_1, \dots, \mathcal B_{l}$ of sidelength $ |h|^{-j-1}
\rho$ with $l \lesssim |h|^{(\frac{d^2+d}{2}-\sum_{i=1}^d
a_i-\frac{(d-j-1)(d-j)}{2})}$. Hence,
\begin{align*}
&\sigma_h^\mba (B(0, \rho)) \le \int \chi_{B(0,\rho)} ( D_h^\mba y)
d\sigma(y) \leq \int \chi_{\mathcal R} (y) d\sigma(y)\\
&\le \sum_{i=1}^l \int \chi_{\mathcal B_i} (y) d\sigma(y)  =  \sum_{i=1}^l \sigma(\mathcal{B}_i)
\lesssim M \sum_{i=1}^l  |h|^{ -(j+1)\alpha} \rho^\alpha\\
& \le C M |h|^{(\frac{d^2+d}{2}-\sum_{i=1}^d
a_i-\frac{(d-j-1)(d-j)}{2}-(j+1)\alpha)}\rho^\alpha.
\end{align*}
This gives the desired inequality since
$(j+1)\alpha+{(d-j-1)(d-j)}/{2}=\beta(\alpha)$.
\end{proof}

\subsection*{\it Multilinear $(d-linear)$ estimates with separated supports} We now prove a multilinear estimate with respect
to general measures, which is basically a consequence of
Plancherel's theorem. We also show that the estimates are uniform
along $\gamma\in \clag$ if $\epsilon>0$ is small enough.

Let us define a map $\Gamma_\gamma: I^d\to \mathbb R^d$ by
\[
\Gamma_\gamma(\mbt)=\sum_{i=1}^d \gamma(t_i),
\]
where $\mbt = (t_1, t_2, \cdots , t_d)$.

\begin{lemma}\label{1-1} Let $E=\{\mbt\in I^d:  \, t_1<
t_2< \ldots < t_d\}$.  If $\epsilon>0$ is sufficiently small, then
the map $\mbt: E\to \Gamma_\gamma(\mbt)$ is one to one for all
$\gamma\in \clag$.
\end{lemma}

This can be shown by the argument in \cite{DM2} which relies
on {\it total positivity} (also see \cite{DeW}). In fact, we need
to show that total positivity is valid on $I$ regardless of
$\gamma\in \clag$. It is not difficult by making use of the fact
that $\gamma$
is a small perturbation of $\gamma_\circ$. 
We give a proof of Lemma \ref{1-1} in Appendix \ref{appendb}.

Now, the following is a straightforward consequence of Plancherel's
theorem.

\begin{lemma}\label{l2} Let $\gamma \in \clag$ and $\mathcal I_1,\dots, \mathcal I_d$
be closed intervals  contained in $I$ which satisfy $\min_{i\neq
j}\dist(\mathcal I_i, \mathcal I_j)\ge L>0$. If $\epsilon>0$ is
sufficiently small, then there is a constant $C$, independent of
$\gamma$, such that
\[
 \| \prod_{i=1}^dT_\lambda^\gamma f_i \|_{L^2}\le CL^{-\frac{d^2-d}{4}}\lambda^{-\frac d 2}   \prod_{i=1}^d  \| f_i \|_{L^2}
\]
whenever $f_i$ is supported in $\mathcal I_i$, $i=1,2,\dots, d$.
\end{lemma}

\begin{proof}
Let $f_1,\dots, f_d$ be supported in each interval  $\mathcal I_i$,
$i=1,2,\dots, d$, and set $F(\mbt)=\prod_{i=1}^d
f_i(t_i)\chi_{\mathcal I_i}(t_i)$. Then, by Lemma \ref{1-1}
$\Gamma_\gamma:\prod_{i=1}^d \mathcal I_i\to \mathbb R^d$ is one to
one. Hence by the change of variables $y=\Gamma_\gamma(\mbt)$, we
have
\[
\prod_{i=1}^d T_\lambda^\gamma f_i = \int_{I^d} e^{i \lambda x\cdot \Gamma_\gamma(\mbt) } F(\mbt) d\mbt
  = \widehat{G} (\lambda x),
\]
where $ G( y ) = F( \mbt (y) ) | \det \big(
\frac{\partial\Gamma_\gamma}{\partial \mathbf t} \big) (\mbt (y))
|^{-1}$. By Plancherel's theorem and reversing the change of
variables $y \mapsto \mbt$, we see
\begin{align*}
\| \prod_{i=1}^d T_\lambda^\gamma f_i \|_{L^2(\Real^d)} &= \lambda^{-
\frac{d}{2} } \| G \|_{L^2 (\Real^d) }
 = \lambda^{- \frac{d}{2} } \bigg(  \int |F(\mbt) |^2
 \Big|\det \Big( \frac{\partial\Gamma_\gamma}{\partial \mathbf t} \Big) \Big|^{-1} d\mbt   \bigg)^{\frac{1}{2}}.
  \end{align*}
 Since $|t_j-t_i|\ge  L$, $i\neq j$, it is sufficient to show that for all $\gamma \in \clag$
\[ \Big|\det \Big( \frac{\partial\Gamma_\gamma}{\partial \mathbf t} \Big)\Big|\ge
\frac1{2\prod_{i=1}^d (i-1)!} \prod_{1\le i<j\le d} |t_j-t_i| \]
 if $\epsilon$ is
sufficiently small. If $\gamma \in \clag$, then
$\gamma=\gamma_\circ+\mathcal E$ and $\|\mathcal E\|_{C^{d+1}(I)}\le
\epsilon$. Hence by a computation with a generalized mean value
theorem we see that
\[
\Big|\det \Big( \frac{\partial\Gamma_\gamma}{\partial \mathbf t}
\Big)\Big|\ge \frac1{\prod_{i=1}^d (i-1)!}\prod_{1\le i<j\le d}
|t_j-t_i| \,\times ( 1 - \epsilon 2^{d-1} d!).
\] Taking a small $\epsilon$
so that $\epsilon< ( 2^{d} d!)^{-1}$, we get the desired estimate.
This completes proof.
\end{proof}

Using Lemma \ref{l2}, we obtain the following $L^p$--$L^q$ estimate
via  interpolation with  $L^1$--$L^\infty$ estimate.

\begin{proposition}\label{l2fractal}
 Let $\mathcal I_1,\dots, \mathcal I_d$, and $\gamma \in \clag$ be given as in Lemma \ref{l2}.
Suppose  $\mu$ satisfies \eqref{bmeasure}. If $\epsilon>0$ is
sufficiently small, then for $1/p+1/q\le 1$ and $q\ge 2$ there is a
constant $C$, independent of $\gamma$, such that
\[
 \| \prod_{i=1}^dT_\lambda^\gamma f_i \|_{L^q(d\mu)}\le C C_\mu^\frac1{q}
 L^{-\frac{d^2-d}{2q}}\lambda^{-\frac\alpha{q}}\prod_{i=1}^d\| f_i \|_p
 \]
 whenever $f_i$ is supported in $\mathcal I_i$, $i=1,2,\dots, d$.
\end{proposition}
\begin{proof} To begin with, we observe that  the  trivial $L^1$--$L^\infty$ estimate
\begin{equation*}
\| \prod_{i=1}^dT_\lambda^\gamma f_i \|_{L^\infty (d\mu)}\le
 \prod_{i=1}^d\| f_i \|_1
\end{equation*}
holds. It is obvious  because $ \prod_{i=1}^dT_\lambda^\gamma f_i $
is continuous and $|\prod_{i=1}^dT_\lambda^\gamma f_i|\le
\prod_{i=1}^d\| f_i \|_1$. Since $f_i$ is supported in $\mathcal
I_i\subset I$, by H\"older's inequality and interpolation, it suffices to show that
\begin{equation}\label{l22}
\| \prod_{i=1}^d \tla f_i \|_{L^2(d\mu)}\le C C_\mu^\frac12
L^{-\frac{d^2-d}{4}}\lambda^{-\frac\alpha 2}  \prod_{i=1}^d\| f_i \|_2.
\end{equation}

Note that the Fourier transform $\mathcal  F({\prod_{i=1}^d
T_\lambda^\gamma f_i})$  of $\prod_{i=1}^d T_\lambda^\gamma f_i$ is
supported in a ball of radius $ C\sqrt{2d}\,\lambda$ for some $C>0$.
Let $\varphi$ be a smooth function such that $\widehat\varphi=0$ if
$|\xi|\ge 2C\sqrt{2d}$  and $\widehat\varphi=1$ if $|\xi|\le
C\sqrt{2d}$.
Then, $\mathcal  F({\prod_{i=1}^d T_\lambda^\gamma
f_i})(\xi)=\mathcal F({\prod_{i=1}^d T_\lambda^\gamma
f_i})(\xi){\widehat{\varphi}}(\xi/\lambda)$. Hence $\prod_{i=1}^d
T_\lambda^\gamma f_i=\varphi_\lambda \ast (\prod_{i=1}^d
T_\lambda^\gamma f_i)$ where
$\varphi_\lambda(x)=\lambda^d\varphi(\lambda x)$. This and
H\"older's inequality gives
\[|\prod_{i=1}^d T_\lambda^\gamma f_i|^2 \le C|\prod_{i=1}^d
T_\lambda^\gamma f_i|^2\ast |\varphi|_\lambda\] with $C$ only
depending on $\varphi$. By the rapid decay of $\varphi$ and
\eqref{bmeasure}, it follows that
\[
|\varphi_\lambda | \ast \mu  = \int \lambda^d |\varphi | (
\lambda(x-y) ) d\mu(y)  \leq C C_\mu \lambda^{d-\alpha}.
\]
Therefore, this and Fubini's theorem give
\begin{align*}
\|\prod_{i=1}^dT_\lambda^\gamma f_i\|_{L^2(d\mu)} &\le \Big(\int
|\prod_{i=1}^d T_\lambda^\gamma f_i|^2 \ast |\varphi|_\lambda
 d\mu(x)\Big)^\frac12
\le \|\prod_{i=1}^d T_\lambda f_i\|_2\,\||\varphi|_\lambda\ast
\mu\|_\infty^\frac12
\\
&\lesssim C_\mu^\frac12 L^{-\frac{d^2-d}{4}} \lambda^{-\frac d2}\lambda^{\frac d2-\frac
\alpha 2}  \prod_{i=1}^d \|f_i\|_2
\lesssim C_\mu^\frac12 L^{-\frac{d^2-d}{4}} \lambda^{-\frac\alpha 2}\prod_{i=1}^d \|f_i\|_2.
\end{align*}
For the third inequality we use Lemma \ref{l2}.
Hence we get \eqref{l22}.
\end{proof}

\subsection*{\it Induction quantity}
For $1\le \lambda,$ $1\le p, q\le \infty$, and $\epsilon>0$, we
define $Q_\lambda(R) =Q_\lambda(R, p, q,\epsilon)$ by setting
\begin{align}\label{aqlambda}
Q_\lambda(R) = \sup\{\,\, \|\tla f\|_{L^q(d\mu,B_R)}: \mu\in
\mathfrak B(\alpha,1),\, \gamma\in \clag, \, {\|f\|_{L^p(I)}\le 1}
 \}
\end{align}
where $B_R$ is the open ball of radius $R$ centred at the origin.
Clearly, $Q_\lambda(R)<\infty$ because $Q_\lambda(R) \le
R^{\alpha/q}$ for any $\lambda>0$.

\begin{lemma}\label{induction} Let $\gamma\in  \clag$, $\mu\in \mathfrak
B(\alpha,1)$, and let $\lambda\ge  1$, $0<|h|<1$. Suppose that $f$
is supported in the interval $[\tau,\tau+h]^*\subset[0,1]$. Then, if
$\epsilon>0$ is sufficiently small,  there is a constant $\delta>0$,
independent of $\gamma$, such that if $0<|h|\le \delta$
\begin{equation}\label{qlambda}
 \|\tla f\|_{L^q(d\mu,B_R)}\le C \,|h|^{1-\frac1p-\frac{\beta(\alpha)}{q}} Q_{\lambda }(R)
 \|f\|_p.
\end{equation}
\end{lemma}

\begin{proof} We begin with setting $f_h=|h| f(ht+\tau)$. By translation, scaling and using
\eqref{normalcurve} it follows that
\begin{align}\label{scaling}
\begin{aligned}
 | \tla f(x) |
 =  \bigg| \int_I e^{i\lambda x \cdot ( \gamma (ht+\tau) - \gamma(\tau) )}  f_h(t) dt \bigg|
 = \bigg| \int_I e^{i\lambda D_h (M_\tau^\gamma)^t x \cdot \gth(t)} f_h(t) dt
 \bigg|.
 \end{aligned}
\end{align}
We denote by $\mu_\tau^h$ the measure given by \begin{equation}
\label{measureth}
 \int F(x) d\mu_\tau^h(x)= \int F(D_h
(M_\tau^\gamma)^t  x) d\mu(x)
\end{equation}
and set
\[d\widetilde \mu(x)=\frac{|h|^{\beta(\alpha)}}{1+ C
\|(M_\tau^\gamma)^{-t}\|^\alpha}\,d\mu_\tau^h(x).\] Hence, we have
\begin{align*}
 &\int_{B_R} | \tla f(x) |^q d\mu(x)
= \int |  T_{\lambda}^{\gth} f_h (x) |^q
\chi_{B_R}((M_\tau^\gamma)^{-t}D_h^{-1} x)
d\mu_\tau^h(x)\\
= &(1+ C \|(M_\tau^\gamma)^{-t}\|^\alpha)|h|^{-\beta(\alpha)} \int |
T_{\lambda}^{\gth} f_h (x) |^q
\chi_{B_R}((M_\tau^\gamma)^{-t}D_h^{-1} x) d\widetilde\mu(x)\chg{.}
\end{align*}

Now we note that $\|(\mga)^{-t}\|\le C$ uniformly for $\gamma\in
\clag$ if $\epsilon>0$ is small enough.  By Lemma \ref{rescale}\,
$\widetilde \mu\in \mathfrak B(\alpha,1)$, and $\gth\in \mathfrak
G(C|h|\epsilon)\subset \clag$ if $0<|h|\le \delta$\, for a
sufficiently small $\delta>0$. Moreover, the set $\{ \mga D_h x :
x\in B_R\}$ is also contained in $B_R$ for all $\gamma\in \clag$ if
$0<|h|\le \delta$ and $\delta$ is small enough. Therefore, by the
definition of $Q_{\lambda }(R)$ we see
\begin{align*}
\int_{B_R} |\tla f(x)|^q d\mu(x) &\le C|h|^{-\beta(\alpha)}
\int_{B_R} | T_{\lambda}^{\gth} f_h (x) |^q d\widetilde\mu(x)
\\
&\le C|h|^{-\beta(\alpha)}(Q_{\lambda}(R) \,\| f_h\|_p) ^q \\
& = C |h|^{-\beta(\alpha)+q-\frac qp} (Q_{\lambda }(R) \| f\|_p)^q.
\end{align*}
This gives the desired inequality \eqref{qlambda}.
\end{proof}

\subsection*{\it Multilinear decomposition} Now we make a
decomposition of $T_\lambda^\gamma$ which is needed to exploit the
$d$-linear estimate with separated supports. This decomposition
doesn't depend on particular choices of $\gamma$.

 Let $A_1,\dots, A_{d-1}$ be dyadic numbers such that
\[1=A_0\gg  A_1\gg A_2 \dots\gg  A_{d-1}.\]
 For $i=1, \dots, d-1$, let us denote by $ \{ \cuxi^i\}$ the
collection of closed dyadic intervals of length $A_i$ which are
contained in $[0,1]$. And we set
\[ f_{\cuxi^i}=\chi_{\cuxi^i} f\]
so that for each $i=1, \dots, d-1,$ $ f=\sum_{\cuxi^i} f_{\cuxi^i} $
almost everywhere whenever $f$ is supported in $I$. Hence, it
follows that \begin{equation}\label{iscale}\tla f=\sum_{\cuxi^i}
\tla f_{\cuxi^i}, \,\, i=1, \dots, d-1.\end{equation}

 Let $S_1,\dots, S_i$ be subsets of $I$ and let us
define
\[\Delta(S_1, S_2, \dots, S_i)=\min_{j\neq k}\, \dist(S_j, S_k).\]

\begin{lemma}\label{multidecomp} Let $\gamma:I\to \mathbb R^d$ be a smooth
curve. Let $A_0,A_1,\dots, A_{d-1}$, and $\{\cuxi^i\}$, $i=1, \dots,
d-1$ be defined as in the above. Then, for any $x\in \mathbb R^d$,
there is a constant $C$, independent of $\gamma, x$, $A_0,A_1,\dots,
A_{d-1}$, such that
 \begin{equation}
 \begin{aligned}\label{1d-decomp}
|\tla & f(x)|\le \sum_{i=1}^{d-1} CA_{i-1}^{-2(i-1)} \max_{\cu i}
|\tla \ef{i}(x)|\\
&+ CA_{d-1}^{-2(d-1)}\max_{\substack{\cuu{d-1}1,\cuu{d-1}2,
\dots,\cuu{d-1}d;\\ \Delta(\cuu{d-1}1, \cuu{d-1}2,
\dots,\cuu{d-1}d)\ge A_{d-1}}} |\prod_{i=1}^d \tla \eff{d-1}i
(x)|^\frac1d.
\end{aligned}
\end{equation}
Here $\cuu{i}j$ denotes the element in $\{\cuxi^i\}$.
\end{lemma}

The exact exponents of $A_i$ are not important for the argument
below. So, we don't try to obtain the  best exponents.

\begin{proof}
Fix $x\in \mathbb R^d$. By a simple argument it is easy to see that
\begin{equation} \label{step1} |\tla f(x)|\le C\max_{\cu1} |\tla \ef{1}(x)|+
CA_1^{-1} \max_{\substack{\cuu11, \cuu12;\\ \Delta(\cuu11,
\cuu12)\ge {A_1}}} |\tla \eff11 (x)\tla\eff12(x)|^\frac12
.\end{equation} Indeed, let $\cu1_\ast\in \{\cuxi^1\}$ be the
interval such that $|\tla f_{\cu1_{\ast}}(x)|= \max_{\cu1}
|\tla\ef{1}(x)|$. We now consider the cases
$ |\tla f(x)|\le 100|\tla f_{\cu1_{\ast}}(x)|,$ $ |\tla f(x)|\ge
100|\tla f_{\cu1_{\ast}}(x)|,$ separately. For the latter case, from
\eqref{iscale} it is easy to see that  there is an interval
$\cu1_1\in \{\cuxi^1\}$ such that $|\tla f(x)|\le C
A_1^{-1} |\tla {f_{\mathfrak I^1_1}}(x)|$ 
and $\Delta(\cu1_\ast, \cu1_1)\ge A_1$.
 Then it follows that
\[|\tla f(x)| \le C A_1^{-1} |\tla {f_{\mathfrak I^1_1}}(x) 
\tla f_{\cu1_{\ast}}(x)|^\frac12.\] Combining  two cases we get the
desired inequality \eqref{step1}, which is clearly independent of
$x$ and $\gamma$.

Now, for $2\le j\le d$ we claim that
\begin{equation}\label{induct}\begin{aligned}
&\max_{\substack{\cuu{j-1}1, \cuu{j-1}2, \dots,\cuu{j-1}j;\\
\Delta(\cuu{j-1}1, \cuu{j-1}2, \dots,\cuu{j-1}j)\ge {A_{j-1}}}}
|\prod_{i=1}^j  \tla \eff{j-1}i (x)|^\frac1j \\
\le& C\max_{\cu j} |T_\lambda \ef{j}(x)| +
CA_{j}^{-2}\max_{\substack{\cuu{j}1, \cuu{j}2, \dots,\cuu{j}{j+1};\\
\Delta(\cuu{j}1, \cuu{j}2, \dots,\cuu{j}{j+1})\ge {A_{j}}}}
|\prod_{i=1}^{j+1} \tla \eff{j}i (x)|^\frac1{j+1}
\end{aligned}\end{equation}
holds with $C$, independent of $x,$ $\gamma$, and $A_{j-1}, A_{j}$.
This proves the desired inequality. In fact, starting from
\eqref{step1} and applying \eqref{induct} successively to the product
terms we see that
\begin{align*}
|\tla & f(x)|\le C\sum_{i=1}^{d-1} \Big[\prod_{k=1}^{i} A_{k-1}^{-2}\Big]\max_{\cu i} |\tla \ef{i}(x)|\\
&+ C\Big[\prod_{k=1}^{d-1}
A_{k-1}^{-2}\Big]\max_{\substack{\cuu{d-1}1, \cuu{d-1}2,
\dots,\cuu{d-1}d;\\\Delta(\cuu{d-1}1, \cuu{d-1}2,
\dots,\cuu{d-1}d)\ge A_{d-1}}} |\prod_{i=1}^d \tla \eff{d-1}i
(x)|^\frac1d.
\end{align*}
This clearly implies \eqref{1d-decomp}. Hence it remains  to show
\eqref{induct}.

Suppose that intervals $\cuu{j-1}1, \cuu{j-1}2, \dots,\cuu{j-1}j\in
\{\cuxi^{j-1}\}$ with $\Delta(\cuu{j-1}1,$ $ \cuu{j-1}2,$ $
\dots,\cuu{j-1}j)$ $\ge {A_{j-1}}$ are given. For each $i
=1,\dots,j$, we denote by $\{\cuu{j}{ \alpha_i}\}$ the collection of
the dyadic intervals which satisfy   $\cuu{j}{ \alpha_i}\subset \cuu
{j-1}  i$ and $\cuu{j}{ \alpha_i}\in\{\cuxi^j\}$.   Clearly, the
number of $\{\cuu{j}{\alpha_i}\}$ is $A_{j-1}/A_j$. Since
$\bigcup_{\cuu{j}{\alpha_i}} \cuu{j}{ \alpha_i}=\cuu{j-1}i$,  \[\tla
\eff{j-1}i= \sum_{\cuu{j}{\alpha_i}} \tla \eff{j}{\alpha_i}.\]
Let $\cuu{j}{ \alpha_{i,\ast}}\in \{\cuu{j}{ \alpha_i}\}$ denote the
dyadic interval
 such that
\[ |\tla \eff{j}{\alpha_{i,\ast}}(x)|=\max_{\cuu{j}{\alpha_{i}}}|\tla \eff{j}{\alpha_{i}} (x)|.\]
Given a $j$-tuple $(\cuu{j}{ \alpha_1}, \dots, \cuu{j}{
\alpha_{j}})$ of intervals, there are the following two cases:
\[
\tag{I} \label{case11}  \text{or }\,\, \begin{aligned}&|\tla
\eff{j}{\alpha_i} (x)|< A_j^{j} \max_{i=1,\dots,j}|\tla
\eff{j}{\alpha_{i,\ast}} (x)| \text{ for some } i,
\\
& 
\Delta(\cuu{j}{\alpha_i}, \cuu{j}{\alpha_{i,\ast}}) < A_{j}\text{
for all } i,
\end{aligned}
\]
and
\[
\tag{I\!I}\label{case2} \text{and }\, \begin{aligned}&|\tla
\eff{j}{\alpha_i} (x)|\ge A_j^{j} \max_{i=1,\dots,j}|\tla
\eff{j}{\alpha_{i,\ast}} (x)| \text{ for all } i ,
\\&
\Delta(\cuu{j}{\alpha_i}, \cuu{j}{\alpha_{i,\ast}}) \ge A_{j} \text{
for some } i.
\end{aligned}
\]

We now split
\begin{align*}
\prod_{i=1}^j  \tla \eff{j-1}i (x) =\sum_{\cuu{j}{\alpha_1},\dots,
\cuu{j}{\alpha_j}} \prod_{i=1}^j  \tla \eff{j}{\alpha_i} (x)
=\Big(\sum_{\eqref{case11}}+\sum_{\eqref{case2}}\Big)\prod_{i=1}^j
\tla \eff{j}{\alpha_i}(x).
\end{align*}
Since $\#\{\cuu{j}{\alpha_i}\}\le A_j^{-1}$, there are $O(A_j^{-j})$
$j$-tuples $(\cuu{j}{\alpha_1},\dots, \cuu{j}{\alpha_j})$ in the
summation of the case \eqref{case11}. Hence it is easy to see that
\begin{equation}\label{case1}\Big|\sum_{\eqref{case11}}\prod_{i=1}^j \tla \eff{j}{\alpha_i} (x)\Big|^{\frac 1 j} \le
C\max_{i=1,\dots,j}|\tla \eff{j}{\alpha_{i,\ast}}(x) |.
\end{equation}

We now consider a term $\prod_{i=1}^j  \tla \eff{j}{\alpha_i} (x)$
from the second case \eqref{case2}. Then  $\Delta(\cuu{j}{\alpha_k},
\cuu{j}{\alpha_{k,\ast}}) $ $ \ge A_{j}$ for some $1\le k\le j$.
Since $ |\tla \eff{j}{\alpha_k}  (x)|$ $\ge A_j^{j}
\max_{i=1,\dots,j}$ $|\tla
\eff{j}{\alpha_{i,\ast}}(x) |$, 
\begin{align*}
\prod_{i=1}^j |\tla \eff j {\alpha_i} (x)| &\le \prod_{i=1}^j |\tla
\eff j {\alpha_{i,\ast}} (x)|^{\frac{1}{j+1}}
\prod_{i=1}^j |\tla \eff j {\alpha_{i,\ast}} (x)|^{\frac{j}{j+1}}\\
&\le A_j^{-\frac{j^2}{j+1}} | \tla \eff j {\alpha_k}
(x)|^{\frac{j}{j+1}} \prod_{k=1}^j |\tla \eff j {\alpha_{i,\ast}}
(x)|^{\frac{j}{j+1}}.
\end{align*}
Recalling that  $\Delta(\cuu{j-1}1,$ $ \cuu{j-1}2,$ $
\dots,\cuu{j-1}j)\ge {A_{j-1}}$ and $\cuu j{\alpha_{i,\ast}}\subset
\cuu {j-1}i $ for $i=1, \dots, j$, we see that $
\Delta(\cuu{j}{\alpha_{1,\ast}}, \dots, \cuu{j}{\alpha_j,\ast},
\cuu{j}{\alpha_k})\ge A_j$ because $\Delta(\cuu{j}{\alpha_k},
\cuu{j}{\alpha_{k,\ast}})\ge A_{j}$ and $\cuu j{\alpha_{k}}, \cuu
j{\alpha_{k,\ast}}\subset \cuu {j-1}k$. Therefore,
\begin{align*}
\prod_{i=1}^j  |\tla \eff{j} {\alpha_i} (x)|^\frac1j  &\le
A_j^{-1}\max_{\substack{\cuu{j}1, \cuu{j}2, \dots,\cuu{j}{j+1};\\
\Delta(\cuu{j}1, \cuu{j}2, \dots,\cuu{j}{j+1})\ge {A_{j}}}}
|\prod_{i=1}^{j+1} \tla \eff{j}i (x)|^{\frac 1{j+1}}.
\end{align*}
Since there are $O({A_j^{-j}})$ $j$-tuples $(\cuu{j}{1},\dots,
\cuu{j}{j})$, it follows that
\[\Big|\sum_{\eqref{case2}}\prod_{i=1}^j  \tla
\eff{j}{\alpha_i} (x)\Big|^{\frac1j} \le CA_j^{-2}\max_{\substack{\cuu{j}1, \cuu{j}2, \dots,\cuu{j}{j+1};\\
\Delta(\cuu{j}1, \cuu{j}2, \dots,\cuu{j}{j+1})\ge {A_{j}}}}
|\prod_{i=1}^{j+1} \tla \eff{j}i (x)|^{\frac 1{j+1}}.\]

 Combining this with  \eqref{case1} we get  \eqref{induct}.
\end{proof}

\subsection*{\it Proof of Theorem \ref{mainthm}}
Since $\gamma\in C^{d+1}(I)$ satisfies \eqref{torsion}, by
continuity it follows that there is  a constant $C_\gamma$ such that
\[ \|(M_\tau^\gamma)^{-1}\|\le C_\gamma, \, \tau\in I. \]
Let  $0<\epsilon\le 1$ be a small number so that Proposition
\ref{l2fractal} and Lemma \ref{induction} holds. Then fix
$0<\delta<1$ such that Lemma \ref{normalization} holds.

Fixing an integer $\ell $ satisfying $1/\ell < \delta$, we now break
the interval $I$ such that $I=\cup_{j=0}^{\ell-1} [\frac j \ell,
\frac{j+1}\ell]$. Then let us set $h=1/\ell$ and
\[
f_j(t)= hf(ht+jh)\chi_{I}.
\]
 Recalling \eqref{normalcur} and
\eqref{measureth}, for $j=0,\dots,\ell-1$ we also set
\[
\gamma_j=\gamma_{jh}^{h},   \,\,\,\, \mu_j=
\frac{1}{C_{\gamma,j,h}}\, \mu_{jh}^{h},
\]
where $\mu_{jh}^{h}$ is defined by \eqref{measureth} and $
{C_{\gamma, j, h}}=(1+C\|(M_{jh}^\gamma)^{-t}\|^\alpha)
h^{-\beta(\alpha)}. $

Now by Lemma \ref{normalization} it follows that $\gamma_j\in \clag$
and by Lemma \ref{rescale}  we see that $\mu_j\in
\mathfrak{B}(\alpha,1)$. Hence, after rescaling (see {Lemma
\ref{rescale}}) we have \begin{equation} \label{decompp} \|\tla f
\|_{L^q(d\mu)}\le \sum_{j=0}^{\ell-1}  \|\tla f\chi_{[jh,
(j+1)h]}\|_{L^q(d\mu)}
=\sum_{j=0}^{\ell-1} {(C_{\gamma,j,h}})^\frac1q
\|T_{\lambda}^{\gamma_j} f_j \|_{L^q(d\mu_j)}. \end{equation}
Therefore for the proof of Theorem \ref{mainthm} it is sufficient to
show \eqref{frac} when $\gamma\in \clag$, $\mu\in \mathfrak
B(\alpha,1)$.

 Let $p\ge 1$, $q\ge 1$ be numbers such that  $d/{q}\le (1-1/p)$, $q\ge 2d$,
and $\beta(\alpha)/q + 1/p < 1$, $q
> \beta(\alpha) + 1$. It is enough  to
consider $q\ge p$. The other case follows by H\"older's inequality.   Let $\gamma\in \clag$, $\mu\in \mathfrak
B(\alpha,1)$, and $f$ be a function supported in $I$ with
$\|f\|_{L^p(I)}=1$ such that
\[ Q_\lambda(R)=Q_\lambda(R,p,q)\le  2 \| T_\lambda^\gamma f\|_{L^q(d\mu,B_R)}.\]

Set $A_0=1$ and let $A_1,$ $\dots$, $A_{d-1}$ be dyadic numbers such
that $\delta\gg A_1\gg A_2 \dots\gg A_{d-1}$. These numbers are
to be chosen later. Then, by recalling \eqref{1d-decomp}, using
Lemma \ref{induction}, and noting $q\ge p$, we see that
\begin{align*}
&\,\,\Big\|\max_{\cu i} |\tla \ef{i}|\Big\|_{L^q(d\mu,B_R)} \le
\Big(\sum_{\cu i} \Big\|\tla
\ef{i}\Big\|_{L^q(d\mu,B_R)}^q\Big)^\frac1q
\\
\le\, &{A_i}^{1-\frac1p-\frac{\beta(\alpha)}{q}} Q_{ \lambda}(R)
\Big( \sum_{\cu i} \|\ef{i}\|_p^q\Big)^\frac1q \le
{A_i}^{1-\frac1p-\frac{\beta(\alpha)}{q}} Q_{ \lambda}(R) \Big(
\sum_{\cu i} \|\ef{i}\|_p^p\Big)^\frac1p
\\
=\,&{A_i}^{1-\frac1p-\frac{\beta(\alpha)}{q}} Q_{ \lambda}(R)
\|f\|_p\,.
\end{align*}
Since there are as many as $O(A_{d-1}^{-d})$ $d$-tuples
$(\cuu{d-1}1,\cuu{d-1}2, \dots,\cuu{d-1}d)$ of intervals, using
Proposition \ref{l2fractal}, we also have
\[
\| \max_{\substack{\cuu{d-1}1,\cuu{d-1}2, \dots,\cuu{d-1}d;\\
\Delta(\cuu{d-1}1, \cuu{d-1}2, \dots,\cuu{d-1}d)\ge A_{d-1}}}
|\prod_{i=1}^d \tla \eff{d-1}i (x)|^\frac1d \|_{L^q(d\mu)}  \le
CA^{-C}_{d-1}\lambda^{-\frac\alpha q} \|f\|_p.
\]
  By
 \eqref{1d-decomp} and combining the above two
estimates, we see that
\begin{align*}
\|\tla f\|_{L^q(d\mu)}\le C \sum_{i=1}^{d-1}
A_{i-1}^{-C}{A_i}^{1-\frac1p-\frac{\beta(\alpha)}{q}} Q_{
\lambda}(R) \|f\|_p + CA^{-C}_{d-1}\lambda^{-\frac\alpha q}\|f\|_p
\end{align*}
holds independent of $\gamma\in \clag$, $\mu\in \mathfrak
B(\alpha,1)$. Taking supremum 
with respect to $f$ with
$\|f\|_p\le 1$, $\gamma\in \clag$,  and $\mu\in \mathfrak
B(\alpha,1)$, we get
\begin{equation*}
Q_\lambda(R)\le C\sum_{i=1}^{d-1}
A_{i-1}^{-C}{A_i}^{1-\frac1p-\frac{\beta(\alpha)}{q}} Q_{
\lambda}(R) + CA^{-C}_{d-1}\lambda^{-\frac\alpha q}
\end{equation*}
from the definition of $Q_\lambda(R)$. This gives
\begin{equation*}
\lambda^{\frac\alpha q} Q_\lambda(R) \le C\sum_{i=1}^{d-1}
A_{i-1}^{-C}{A_i}^{1-\frac1p-\frac{\beta(\alpha)}{q}}
\lambda^\frac\alpha q Q_{ \lambda}(R) + CA^{-C}_{d-1}.
\end{equation*}
Since $1-\frac1p-\frac{\beta(\alpha)}{q}>0$, we can successively
choose $A_1, \dots, A_{d-1}$  so that
$CA_{i-1}^{-C}{A_i}^{1-\frac1p-\frac{\beta(\alpha)}{q}}$
$<\frac{1}{2d}$ for $i=1,\dots, d-1$. Hence  we get
\begin{equation*}
\lambda^{\frac\alpha q} Q_\lambda(R) \le \frac12\lambda^{\frac\alpha
q} Q_\lambda(R)+ CA^{-C}_{d-1}
\end{equation*}
whenever $ \lambda \ge 1$.  So, it follows that $
\lambda^{\frac\alpha q} Q_\lambda(R) \le CA^{-C}_{d-1}$. Therefore $
Q_\lambda(R) \le C \lambda^{-\frac \alpha q}$.  Letting $R
\rightarrow \infty$ completes the proof. \qed

\begin{rmk}\label{uniform} Note that the estimates in Proposition \ref{l2fractal} and Lemma \ref{induction} hold uniformly  for
all $\gamma\in \clag $,  $\mu\in \mathfrak B(\alpha,1)$ if
$\epsilon>0$ is sufficiently small, and Lemma \ref{multidecomp}
remains valid regardless of particular $\gamma \in \clag$. Hence the
last part of the proof of Theorem \ref{mainthm} actually shows that
there is a constant $C$, independent of $\gamma,$ $\mu$, such that
\[\| T_\lambda^\gamma f\|_{L^q(d\mu)} \le C\lambda^{-\frac \alpha q}
\|f\|_{L^p(I)}\] provided that $\gamma\in \clag $,  $\mu\in
\mathfrak B(\alpha,1)$ and $\epsilon>0$ is sufficiently small.
\end{rmk}

\section{Proof of Theorem \ref{finitethm}; Finite type curves}

 As in the nondegenerate case,
the curve of finite type may be considered  as a perturbation of a
monomial curve in a sufficiently small neighborhood. The following
is a simple consequence of Taylor's theorem.

\begin{lemma}\label{monomial} Let $ \gamma:I=[0,1] \to \mathbb
{R}^d$ be a smooth curve. Suppose that $\gamma$ is of finite type at
$\tau\in I$. Then there exist $\delta>0$, a $d$-tuple $\mba =
(a_1,\cdots,a_d)$ of positive integers satisfying $a_1<a_2< \cdots
<a_d$ such that $\ma{\tau}$ $($defined by
 \eqref{defm}$)$ is  nonsingular and
\begin{equation}
\label{canonic}
 \gamma
(t+\tau)-\gamma(\tau)=\ma{\tau}
(t^{a_1}\varphi_1(t),t^{a_2}\varphi_2(t),\cdots,t^{a_d}\varphi_d(t)),
\end{equation}
 for $t\in [-\delta,\delta]\cap (I-\tau)$ where  $\varphi_k$ is a smooth
function satisfying
 \begin{equation}\label{derivative}
(t^{a_k}\varphi_k)^{(a_j)}(0)=\delta_{jk} \text{ for } 1\le
j\le  k\le d\,.
\end{equation}
\end{lemma}

The last condition \eqref{derivative} implies that $\varphi_k (0) =
1/(a_k!)$ for $1\leq k \leq d$. Furthermore it is easy to see that
$\mba$ and $\varphi_1(t), \dots, \varphi_d(t)$ are uniquely
determined. To see this,
suppose that
\[M(t^{a_1}\varphi_1(t),\cdots,t^{a_d}\varphi_d(t))=
M'(t^{b_1}\widetilde\varphi_1(t),\cdots,t^{b_d}\widetilde\varphi_d(t))\]
for nonsingular matrices $M, M'$, positive integers $b_1<
b_2<\dots<b_d$ and smooth $\widetilde\varphi_i$ with
$(t^{b_k}\widetilde\varphi_k)^{(b_j)}(0)=\delta_{jk}$ for $1\le j\le
k\le d$.   Now let $M_k$ and $M_k'$ denote the $k$-th column of
matrices $M$ and $M'$, respectively. The above is written as
\[ M_1t^{a_1}\varphi_1(t)+\dots+ M_dt^{a_d}\varphi_d(t)=M_1't^{b_1}\widetilde\varphi_1(t)
+\dots+M_d't^{b_d}\widetilde\varphi_d(t). \] Differentiating $a_1$
times at $t=0$, it is easy to see $b_1\le a_1$. By symmetry we also
have $b_1\ge a_1$. Hence $a_1=b_1$ and by \eqref{derivative} we see
that $M_1=M_1'$. Similarly, by differentiating $a_2$ times at $t=0$
and using \eqref{derivative} it follows that $a_2=b_2$  and
$M_2=M_2'$. By repeating this we see that $a_1=b_1$, $\dots,$
$a_d=b_d$ and $M_1=M'_1$, $\dots,$ $M_d=M'_d$. Then,  since
$M_1,\dots, M_d$ are linearly independent,   it follows that
$\varphi_i(t)=\widetilde\varphi_i(t)$ for $i=1,\dots, d$.

Therefore, thanks to Lemma \ref{monomial} we can have the following
definition.

\begin{defn} Let $ \gamma:I=[0,1] \to \mathbb
{R}^d$ be a smooth curve and $\tau\in I$. If there are a nonsingular
matrix $M$, positive integers $a_1, a_2, \cdots, a_d$ with $a_1<a_2<
\cdots <a_d$ and smooth functions $\varphi_1,\dots, \varphi_d$
satisfying \eqref{derivative} such that
\begin{equation*}
\gamma
(t+\tau)-\gamma(\tau)=M(t^{a_1}\varphi_1(t),t^{a_2}\varphi_2(t),\cdots,t^{a_d}\varphi_d(t)),
\,\, t\in [-\delta,\delta]\cap (I-\tau)
\end{equation*}
for some $\delta>0$, then we say that $\gamma$ is of type $\mba$ at
$t=\tau$.
\end{defn}

\begin{proof}[Proof of Lemma \ref{monomial}] Let $a_1$ be the smallest integer such that
$\gamma^{(a_1)}(\tau)\neq 0$. And let $a_2$  be the smallest integer
such that $\gamma^{(a_1)}(\tau)$ and $\gamma^{(a_2)}(\tau)$ are
linearly independent. Then, we inductively choose $a_{j}$ to be the
smallest integer such that $\gamma^{(a_1)}(\tau),$ $ \dots,$ $
\gamma^{(a_{j-1})}(\tau),\gamma^{(a_{j})}(\tau) $ are linearly
independent. Since $\gamma$ is finite type at $\tau$, this gives
linearly independent vectors $\gamma^{(a_1)}(\tau),
\dots,\gamma^{(a_{d})}(\tau)$.

Let us set $a_0=0$. Then for $j=1, \dots, d$, it follows that if
$a_{j-1}<\ell<a_{j}$
\begin{equation} \label{span}\gamma^{(\ell)}(\tau)\in
\text{span}\{\gamma^{(a_1)}(\tau), \dots,
\gamma^{(a_{j-1})}(\tau)\}.\end{equation}
By Taylor expansion of
$\gamma(t)$ at $t=\tau$, we  write
\[
\gamma(t+\tau) = \gamma(\tau) + \sum_{\ell=1}^{a_d}
\frac{t^{\ell}}{\ell !} \gamma^{(\ell)}(\tau)+
 \frac{t^{a_d+1}}{(a_d+1)!}\, \mathcal E (\mathbf{c})
\]
where $\mathcal E (\mathbf{c}) =
(\gamma_1^{(a_d+1)}(c_1),\dots,\gamma_d^{(a_d+1)}(c_d))$ and
$\mathbf{c} = (c_1,\dots,c_d)$ with $c_i \in (\tau,t+ \tau)$, $1\leq
i \leq d$. Then by \eqref{span} it follows that for $j=2,\dots, d$
\[\sum_{\ell=a_{j-1}}^{a_j-1}
\frac{t^{\ell}}{\ell !}\gamma^{(\ell)}(\tau)
=\frac{t^{a_{j-1}}}{a_{j-1}!}\gamma^{(a_{j-1})}(\tau)+
\sum_{k=1}^{j-1} p_{j,k}(t) \gamma^{(a_k)}(\tau)\] with polynomials
$p_{j,k}(t)$ which consist of monomials of degree $ d_{j,k}$,
$a_j-1\ge d_{j,k}\ge a_{j-1}+1$. Also, $\mathcal E (\mathbf{c})$ is
obviously spanned by $\gamma^{(a_1)}(\tau),$ $ \dots,$
$\gamma^{(a_{d})}(\tau) $ so that $ \frac{t^{a_d+1}}{(a_d+1)!}
\mathcal E (\mathbf{c})= e_{1}^*(\mathbf{c})\,
t^{a_d+1}\gamma^{(a_1)}(\tau)+\dots
+e_{d}^*(\mathbf{c})\,t^{a_d+1}\gamma^{(a_{d})}(\tau)$. Therefore
\[
\gamma(t+\tau) = \gamma(\tau) + \sum_{k=1}^{d}
\Big(\frac{t^{a_k}}{a_k !} + p_k(t)+ e_{k}^*(\mathbf{c}) t^{a_d+1}\Big)
\gamma^{(a_k)}(\tau)
\]
where $p_k$ is a polynomial  which consists of monomials of degree
$\ell $, $a_k\le \ell\le a_d$ and $\ell\not\in\{a_{k},
a_{k+1},\dots, a_{d}\}$. We now set
\[t^{a_k}\varphi_k(t)=(\frac{t^{a_k}}{a_k !} + p_k(t)+ e_{k}^*(\mathbf{c})
t^{a_d+1}).\] Then \eqref{canonic} follows and \eqref{derivative} is
easy to check. This completes the proof.
\end{proof}


\subsection*{\it Normalization of finite type curves.} Let $\mba = (a_1,\cdots,a_d)$ be
a $d$-tuple of positive integers satisfying $a_1<a_2< \cdots <a_d$.
For $\epsilon > 0$, let us define  the class $\mathfrak{G}^\mba
(\epsilon)$ of smooth curves by setting
\begin{align*}
\mathfrak{G}^\mba(\epsilon) = \Big\{ \gamma \in C^\infty (I) :
\gamma(t)&=(t^{a_1}\varphi_1(t),t^{a_2}\varphi_2(t),\cdots,t^{a_d}\varphi_d(t)), \\
&\Big\|\varphi_i - \frac{1}{a_i!} \Big\|_{C^{a_d+1}(I)} \leq
\epsilon\Big\}.
\end{align*}
Let
$\gamma$ be of type $\mba$ at $\tau$. 
Recalling \eqref{defm} and \eqref{normalcur}, for
$[\tau,\tau+h]^*\subset I$ let us set
\begin{equation}\label{scaled}
\gamma_{\tau}^{h,\mba}(t)= [\ma \tau D_h^\mba
]^{-1}(\gamma(ht+\tau)-\gamma(\tau)).
\end{equation}
Here $D_h^\mba$ is given by \eqref{diagonal}.
Then by Lemma \ref{monomial} it
follows that
\begin{equation}\label{finitenormal}
\gatha= (t^{a_1}\varphi_1(ht), t^{a_2} \varphi_2(ht), \dots,
t^{a_d}\varphi_d(ht))
\end{equation}
for  $\varphi_1, \dots, \varphi_d$ which are smooth functions on $I$
and  satisfy $\eqref{derivative}$. Hence, it is easy to see the
following.

\begin{lemma}\label{finite}   Let $ \gamma:I=[0,1] \to \mathbb
{R}^d$ be a smooth curve. Suppose that $\gamma$ is of type $\mba$ at
$\tau$. Then for any $\epsilon>0$ there is an
$h_\circ=h_\circ(\mba,\epsilon,\tau)>0$ such that $\gatha\in
\mathfrak{G}^\mba(\varepsilon)$ if $[\tau,\tau+h]^*\subset I$ and
$0<|h|<h_\circ$.
\end{lemma}

The curves in $\mathfrak{G}^\mba(\varepsilon)$ are clearly close to
the curve \[\gamma^\mba_\circ =\Big( \frac{t^{a_1}}{a_1!},\,\dots\,
,\frac{t^{a_d}}{a_d!}\Big).\] Hence,  as in the nondegenerate case
the upper and lower bounds of the torsion can be controlled
uniformly as long as the curve belongs to
$\mathfrak{G}^\mba(\varepsilon)$. The following is a slight variant
of Lemma 2 in \cite{DeM}.

\begin{lemma}\label{finitetorsion}
 Let $\gamma(t) \in  \mathfrak{G}^\mba(\varepsilon)$. If $\varepsilon>0$ is sufficiently small,
 then there is a constant $B>0$, independent of $\gamma$, such that
\begin{equation*}
(B/2)\, t^{\sum_{i=1}^d a_i - \frac{d(d+1)}{2}}\leq \det(\gamma'(t)
, \cdots , \gamma^{(d)} (t) ) \leq  2B\, t^{\sum_{i=1}^d a_i -
\frac{d(d+1)}{2}}.
\end{equation*}
\end{lemma}

\begin{proof}
Let us set \[\Phi_{i,j}(t) = \sum_{k=0}^{j-1}
a_i(a_i-1)\cdots(a_i-(j-k-1)) \binom{j}{k} t^k \varphi_i^{(k)}(t) +
t^j \varphi_i^{(j)}(t).\] Then it is easy to see that
$\frac{d^j}{dt^j}(t^{a_i}\varphi_i(t))=t^{a_i-j}\Phi_{i,j}$. So,
the torsion of $\gamma(t)$ can be written as
\begin{align*}
&\det (\gamma'(t) , \cdots,\gamma^{(d)}(t) )
\\
& =  t^{\sum_{i=1}^d a_i - \frac{d(d+1)}{2}} \det \begin{pmatrix}
\Phi_{1,1}(t) & \Phi_{1,2}(t)&\cdots &\Phi_{1,d}(t)\\
\vdots&\vdots&\ddots&\vdots
\\
\Phi_{d,1}(t) & \Phi_{d,2}(t)&\cdots &\Phi_{d,d}(t)\end{pmatrix}
\\
&=: t^{\sum_{i=1}^d a_i - \frac{d(d+1)}{2}}\det\Phi(t).
\end{align*}
Since
$(t^{a_1}\varphi_1(t),t^{a_2}\varphi_2(t),\cdots,t^{a_d}\varphi_d(t))\in
\mathfrak{G}^\mba(\varepsilon)$, it follows that
\[\Phi_{i,j}(t) = \Phi_{i,j}(0)+O(\epsilon)=\frac{\prod_{l=0}^{j-1}(a_i - l)}{a_i!}
+ O(\epsilon).\] So, $\det\Phi(t)=\det\Phi(0)+O(\epsilon)$. Hence if
$\epsilon$ is sufficiently small, $\frac12 \det\Phi(0)\le
\det\Phi(t)\le 2\det\Phi(0)$. This gives the desired inequality.
\end{proof}

\begin{rmk}\label{positivity}
This lemma holds for any minor of the matrix $(\gamma'(t) ,\gamma''(t)
,\cdots,\gamma^{(d)}(t))$. In fact, if a $k \times k$ submatrix  $M_k$
contains $i_1, \dots, i_k$-th rows of $(\gamma'(t) ,\gamma''(t)
,\cdots,\gamma^{(d)}(t))$, then $\det(M_k)$ is bounded above and below
by $t^{\sum_{l=1}^k (a_{i_l} -i_l)}$ uniformly for $\gamma\in \mathfrak
G^\mba(\epsilon)$ if $\epsilon>0$ is sufficiently small.
\end{rmk}

\subsubsection*{Normalization via scaling} We now start proof of Theorem \ref{finitethm}.  Fix $0<\alpha\le
d$ and set \[\sigma = \frac1{\beta(\alpha)}\Big(\sum_{i=1}^d a_i
-\frac{d(d+1)}2\Big)+1.\]
Let $\gamma$ be a finite type curve defined on $I$ and
$[\tau,\tau+h]^*\subset I$.  We consider the integral
\[T_\tau^h f(x)=\int_{[\tau,\tau+h]^*} e^{i \lambda x \cdot \gamma(t) } f(t) \wei (t)
dt.
\]
Let us set $f_\tau^h(t) = f(ht+\tau)$. By changing variables $t\to
ht + \tau$ and \eqref{scaled}, it follows that
\begin{equation}\label{scalinga}\begin{aligned} |T_\tau^h f(x)|&= \Big |\int_{[\tau,\tau+h]^*}
e^{i \lambda x \cdot
(\gamma(t)-\gamma(\tau)) } f(t) \wei (t) dt\Big|\\
 &= |h|\Big| \int_{I} e^{i \lambda D_h^\mba (\ma \tau )^t  x \cdot \gatha} f_\tau^h(t)
 w_{\gamma}^\alpha (ht+\tau)  dt \Big| \end{aligned}\end{equation}
By \eqref{scaled} we get
\begin{equation}\label{weightscale}
|\det(\ma \tau)|^\frac1{\beta(\alpha)} |h|^{\sigma-1}
w_{\gathaa}^\alpha(t)= w_{\gamma}^\alpha (ht+\tau).\end{equation}
Hence, combining this with \eqref{scalinga} we have
\begin{align}\label{rescaledktau}
|T_\tau^h f(x)|= |\det(\ma \tau)|^\frac1{\beta(\alpha)} |h|^\sigma
\big| {T}_\lambda^{\gathaa} [ w_{\gathaa}^\alpha, f_\tau^h]
(D_h^\mba (\ma \tau)^t x) \big|.
\end{align}

 By Lemma \ref{finite}, for  $\tau
\in I$ and $\epsilon>0$, there are $\mba=\mba(\tau)$ and
$h_\circ=h_\circ(\tau,\epsilon)$ such that $\gathaa\in \mathfrak
G^\mba(\epsilon)$ provided that $[\tau, \tau+h]^*\subset I$ and
$0<|h|\le h_\circ$. Since $I$ is compact, we can obviously decompose
the interval $I$ into finitely many intervals of
disjoint interiors
so that $I=\cup_{j=0}^N [\tau_j, \tau_j+h_j]^*$ and
$\gamma_j=\gamma_{\tau_j}^{h_j,\mba_j}\in \mathfrak
G^{\mba_j}(\epsilon)$. By \eqref{scalinga} and
\eqref{rescaledktau} we see that
\begin{equation}
\label{breaking}
\begin{aligned} &|\tlaww{\wei}(x)|\le
\sum_{j=0}^N|T_{\tau_j}^{h_j} f(x)|
\\=&\sum_{j=0}^N
|\det(M^{\gamma,\mba_j}_{\tau_j})|^\frac1{\beta(\alpha)}
{|h_j|}^\sigma \big| {T}_\lambda^{\gamma_j} [ w_{\gamma_j}^\alpha,
f_{\tau_j}^{h_j}] ( D_{h_j}^{\mba_j}(M^{\gamma,\mba_j}_{\tau_j})^t
x) \big|.
\end{aligned}
\end{equation}
Since there are only finitely many terms, in order to show Theorem
\ref{finitethm} it is enough to consider $\mu\in \mathfrak
B(\alpha,1)$ and $\gamma\in \mathfrak G^\mba(\epsilon)$ for some
$\mba$ and a small enough $\epsilon>0$. In fact, define a measure by
$\int F(x)d\widetilde \mu_j=\int F((M^{\gamma,\mba_j}_{\tau_j})^t
D_{h_j}^{\mba_j} x) d\mu.$ By Lemma \ref{rescale} $\widetilde \mu_j$
satisfies \eqref{bmeasure} with some constant $C_{\widetilde \mu_j}$
since $\det(M^{\gamma,\mba_j}_{\tau_j})\neq 0$. Hence, if we set
$\mu_j=(1+C_{\widetilde \mu_j})^{-1}\widetilde \mu_j$, then
$\mu_j\in \mathfrak B(\alpha,1)$. On the other hand, from
\eqref{breaking} we have
\begin{align*}
 \| \tlaww {\wei} \|_{L^q(d\mu)} \le C \sum_{j=0}^N |\det(M^{\gamma,\mba_j}_{\tau_j})|^\frac1{\beta(\alpha)}
{|h_j|}^\sigma\|{T}_\lambda^{\gamma_j} [ w_{\gamma_j}^\alpha,
f_{\tau_j}^{h_j}] \|_{L^q(d\mu_j)}.
\end{align*}

Suppose that \eqref{tlaw} holds for $\mu\in \mathfrak B(\alpha,1)$
and $\gamma\in \mathfrak G^\mba(\epsilon)$ provided that
$\epsilon>0$ is small enough. Then we have $\|{T}_\lambda^{\gamma_j}
[ w_{\gamma_j}^\alpha, f_{\tau_j}^{h_j}] \|_{L^q(d\mu_j)}\le C
\lambda^{-\frac \alpha q} \| f_{\tau_j}^{h_j} \|_{L^p
(w_{\gamma_j}^\alpha dt)}$. So, we get
\begin{align*}
 \| \tlaww {\wei} \|_{L^q(d\mu)} \le \lambda^{-\frac \alpha q}
\sum_{j=0}^N
|\det(M^{\gamma,\mba_j}_{\tau_j})|^\frac1{\beta(\alpha)}
{|h_j|}^\sigma  \| f_{\tau_j}^{h_j} \|_{L^p (w_{\gamma_j}^\alpha
dt)}
\end{align*}
By changing the variables $t\to ({t-\tau_j})/h_j$ and
\eqref{weightscale} it is easy to see that $\| f_{\tau_j}^{h_j}
\|_{L^p (w_{\gamma_j}^\alpha dt)}$ $=
|\det(M^{\gamma,\mba_j}_{\tau_j})|^{-\frac1{\beta(\alpha)}}
{|h_j|}^{-\sigma}  \| f \|_{L^p (w_{\gamma}^\alpha dt)}.$ Hence we
get the desired inequality.

Therefore, it suffices to show that \eqref{tlaw} holds for
$\gamma\in \mathfrak G^\mba(\epsilon)$ and $\mu\in \mathfrak
B(\alpha,1)$ if $\epsilon>0$ is sufficiently small. This will be
done in what follows.

\subsection*{\it Proof of \eqref{tlaw} when $\gamma\in
\mathfrak G^\mba(\epsilon)$ and $\mu\in \mathfrak B(\alpha,1)$.}  We
start with breaking $T_\lambda^\gamma [\wei, f]$ dyadically so that
\[T_\lambda^\gamma [\wei, f](x)= \sum_{j=0}^{\infty} T_j f,\] where
\[T_j f=\int_{[2^{-j-1},2^{-j}]} e^{i \lambda x \cdot \gamma(t) } f(t) \wei
(t) dt.\] In order to prove \eqref{tlaw} it is sufficient to show
that \begin{equation}\label{dyadic}\|T_j f \|_{L^q(d\mu)}\le C 2^{-j
\sigma(1- \frac{\beta(\alpha) }{q}-\frac1p )} \lambda^{-\frac\alpha
q}\|f\|_{L^p(\wei dt)}.\end{equation}

 Let us set
 \[\int F(x)
d\mu_j(x)= \frac{2^{ -j\beta(\alpha)\sigma }}{1+C \| (\ma
0)^{-t}\|^\alpha}\int F( D_{2^{-j}}^\mba (M_{0}^{\gamma,\mba})^t x)
d\mu(x).\] Then $\mu_j\in \mathcal B(\alpha,1)$ by Lemma
\ref{rescale}.
 By rescaling as before ({\it cf.} \eqref{scalinga}) it follows that
\begin{equation*}
\|T_j f \|_{L^q(d\mu)}\le C 2^{-j \sigma(1- \frac{\beta(\alpha) }{q}
)}\|\mathcal T_j f_j \|_{L^q(d\mu_j)} \end{equation*} where
$f_j(t)=f({2^{-j}}t)$ and
\[\mathcal T_j g= \int_{[\frac12,\, 1]} e^{i \lambda  x \cdot \gamma_{0}^{2^{-j},\mba}(t) }
g (t) [2^{(\sigma-1) j}\wei (2^{-j}t)] dt.\]
By rescaling it is easy to see that $
\gamma_{0}^{2^{-j},\mba}\in \mathfrak G^\mba(C2^{-j}\epsilon)$. If
$\epsilon>0$ is small enough,
 by Lemma \ref{finitetorsion} it follows that
 $B_1t^{\sigma -1}\le \wei (t)\le
B_2t^{\sigma -1},$ $t\in I$ with $B_1$, $B_2$, independent of
$\gamma\in \mathfrak G^\mba(\epsilon)$. Hence,  if $\epsilon>0$ is
sufficiently small, then
\[\wei (2^{-j}t) \sim 2^{-(\sigma -1)j} \sim 2^{-(\sigma -1)j} \wei
(t), \,\, t\in [1/2,1] \] with the implicit constants independent of
$\gamma$ as long as $\gamma\in \mathfrak G^\mba(C2^{-j}\epsilon)$.
So, we may disregard the weight. Therefore, for \eqref{dyadic} it is
enough to show uniform estimate $\|\mathcal T_j g
\|_{L^q(d\mu_j)}\le C\lambda^{-\frac\alpha q}\|g\|_{L^p}$ for all
$j\ge 0$. Since $ \gamma_{0}^{2^{-j},\mba}\in \mathfrak
G^\mba(C2^{-j}\epsilon)$, it is sufficient to show that if
$\gamma\in \mathfrak G^\mba(\epsilon)$ and $\mu\in \mathcal
B(\alpha, 1)$, there is a uniform constant $C$ such that
\begin{equation}\label{star}
\|T_\ast^\gamma f \| _{L^q(d\mu)} \le C \lambda^{-\frac\alpha q}\| f
\|_{L^p},\end{equation} where
\[T_\ast^\gamma f(x)=\int_{[\frac12,\, 1]} e^{i \lambda x \cdot \gamma(t) } f(t) dt.\]

Obviously the curve $\gamma\in \mathfrak G^\mba(\epsilon)$ is
uniformly non-degenerate  on $[\frac 1 2,1]$. More precisely let
$\gamma\in \mathfrak G^\mba(\epsilon)$,  $[\tau,\tau+h]\subset
[\frac12,1]$ and consider the curve $ \gamma_\tau^{\,h}$ which is
given by
\[
\gamma_{\tau}^{\,h}(t)= [M_\tau^\gamma D_h
]^{-1}(\gamma(ht+\tau)-\gamma(\tau)).
\]
Since $\tau\in [\frac12, 1]$ and $\gamma\in \mathfrak
G^\mba(\epsilon)$, it follows that $ \|(M_\tau^\gamma)^{-1}\|\le C$
if $\epsilon>0$ is small enough. Hence, by following the argument in
the proof Lemma \ref{normalization} it is easy to see that there is
an $h_0$, independent of $\gamma$, such that $\gamma_{\tau}^{\,h}\in
\clag$ if $h\le h_0$ and $[\tau,\tau+h]\subset[\frac12,1]$ (see
remark \ref{interval}). Hence, we may repeat the lines of argument
in the first part of \emph{Proof Theorem \ref{mainthm}}. In fact,
breaking the interval $[\frac12, 1]$ into $O(1/{h_0})$ essentially
disjoint intervals, by normalization via translation and rescaling
we see that $\|T^\gamma_* f\|_{L^q(d\mu)}$ is bounded by a sum of as
many as $O(1/{h_0})$ of $C\|T_\lambda^{\widetilde\gamma}
f\|_{L^q(d\widetilde \mu)}$ while $\widetilde \gamma\in \mathfrak
G(\epsilon)$ and $\widetilde \mu\in \mathfrak B(\alpha, 1)$ ({\it
cf.} \eqref{decompp}). Finally, from Remark \ref{uniform} we see
that  if $\epsilon>0$ is sufficiently small there is a uniform
constant, independent of $\widetilde \gamma$ and $\widetilde \mu$,
such that $\|T_\lambda^{\widetilde\gamma} f\|_{L^q(d\widetilde
\mu)}\le C\lambda^{-\frac\alpha q}\|f\|_p$ whenever $\widetilde
\gamma\in \mathfrak G(\epsilon)$ and $\widetilde \mu\in \mathfrak
B(\alpha, 1)$.  Therefore we get \eqref{star}. This completes the
proof.

\begin{rmk}\label{remark4nonomial} Since we only rely on scaling and stability of the estimates for the nondegenerate
curves, the argument here also works for the monomial type curves
which were considered in \cite{DeM}. In fact, let $0<a_1< \dots<a_d$
be real numbers and suppose that $ \gamma(t)= (t^{a_1} \varphi_1(t),
\dots, t^{a_d}\varphi_d(t))$, $\varphi_i(0) \neq 0$ and $\lim_{t
\rightarrow 0} t^k \varphi_i^{(k)}(t) = 0$ for $k=1,\dots, d$. Then,
 if  $d/{q}\le (1-1/p)$, $q\ge
2d$, $ \beta(\alpha)/q+1/p<1$  and  $q> \beta(\alpha)+1$, for a
sufficiently small $\delta>0$ the following estimate holds;
\[\Big\|\int_{0}^\delta  e^{i \lambda x \cdot \gamma(t) } f(t) \wei(t) dt\Big\|_{L^q(d\mu)}\le C\|f\|_{L^p(\wei dt)}.\]
\end{rmk}

\appendix

\section{A necessary condition for the estimates \eqref{frac} and \eqref{tlaw}}
 We show that  \eqref{frac} and \eqref{tlaw} hold only if
\begin{equation}\label{necc}
\beta(\alpha)/q+1/p\le 1.
\end{equation}
It is sufficient to consider \eqref{tlaw} since \eqref{frac} is a
special case of \eqref{tlaw}.
 To see this let us fix $j$ so that $d-j-1 < \alpha \leq d - j$.
We consider a measure $\mu$ which is defined by
\[
d\mu (x) = \prod_{i=1}^{j} d\delta(x_i) \, |x_{j+1}|^{\alpha - d +
j}\, dx_{j+1} dx_{j+2}\cdots dx_d.
\]
Here $\delta$ is the delta measure. Then it follows that
$\int_{B(x,\rho)} d\mu(x) \leq C \rho^{\alpha - d +j+1} \cdot
\rho^{d-j-1} = C\rho^\alpha$, i.e. \eqref{bmeasure} is satisfied.
Now let $\gamma(t)$ be a curve of finite type $\mba$ at $\tau$. So,
$\ma\tau$ is nonsingular. We choose $h>0$ small enough so that
$\gathaa\in \mathfrak G^\mba(\epsilon)$ for a small $\epsilon.$ We
define a measure $\widetilde \theta$ by
\[ \int F(x) d\widetilde{\theta}(x) = \int F(
(M_1^{\gamma_\circ^\mba})^{-t}(\ma\tau)^{-t} (D^\mba_h)^{-1} x)d\mu(
x).\] It is easy to show that $d\widetilde{\theta}$ also satisfies
\eqref{bmeasure} with some constant $C_{\widetilde \theta}$.

By taking $f(t) =\chi_{[\tau +h-h\lambda^{-\frac1d},\,\tau+h ] }(t)$
(see \eqref{scalinga}) and changing  variables $t \to h t +\tau$ we
have $ | \tlaww \wei (x) |  = \Big| h \int_{1-\lambda^{-\frac1d}}^1
e^{i \lambda D_h^\mba (\ma \tau )^t x \cdot \gatha}
 w_{\gamma}^\alpha (ht+\tau)  dt \Big|.
$ Then it follows that
\begin{align*} &\| \tlaww \wei
\|_{L^q(d\widetilde\theta)}^q = h^q \int \Big|
\int_{1-\lambda^{-\frac1d}}^1 e^{i \lambda
 x \cdot (M_1^{\gamma_\circ^\mba})^{-1}\gatha}
 w_{\gamma}^\alpha (ht+\tau)  dt \Big|^q d\mu(x)\\&
 = h^q \int \Big|
\int_{-\lambda^{-\frac1d}}^0 e^{i \lambda x \cdot
(M_1^{\gamma_\circ^\mba})^{-1}[\gamma^{h,\mba}_\tau(t+1)-\gamma^{h,\mba}_\tau(1)]}
 w_{\gamma}^\alpha (ht+h+\tau)  dt \Big|^q d\mu(x).
 \end{align*}
By \eqref{weightscale} and Lemma \ref{finitetorsion}, $\wei(h
t+\tau) = |\det(\ma\tau)|^{\frac{1}{\beta(\alpha)}} h^{\sigma-1}
w_{\gamma_\tau^{h,\mba}}^\alpha(t)\sim h^{\sigma-1}|t|^{\sigma - 1}$.
Note that $\gamma_\tau^{h,\mba}$ is nondegenerate on the interval
$[\frac12,1]$ since $\gamma_\tau^{h,\mba}$ is close to
$\gamma_\circ^\mba$ by \eqref{finitenormal}.   By Taylor's expansion
 ({\it cf.} Lemma \ref{normalization}),
$(M_1^{\gamma_\circ^\mba})^{-1}[\gamma^{h,\mba}_\tau(t+1)-\gamma^{h,\mba}_\tau(1)]=\gamma_\circ(t)+O(t^{d+1})$.
Hence it is easy to see that
\[\Big|
\int_{-\lambda^{-\frac1d}}^0 e^{i \lambda x \cdot
(M_1^{\gamma_\circ^\mba})^{-1}[\gamma^{h,\mba}_\tau(t+1)-\gamma^{h,\mba}_\tau(1)]}
 w_{\gamma}^\alpha (ht+h+\tau)  dt \Big|\gtrsim h^{\sigma-1}\lambda^{-\frac1d}\] if $x\in
\mathcal R=\{x=(x_1,\dots,x_d): |x_i|\le c\lambda^{\frac id-1}\}$
with a small $c>0$. Also note that  $ \| f \|_{L^p( \wei dt)}
 = (h \int_{-\lambda^{-1/d}}^0 \wei(ht+h+\tau) dt)^{1/p}
 \sim (h\int_{-\lambda^{-1/d}}^0 h^{\sigma-1}|t+1|^{\sigma-1} dt)^{1/p}
 \lesssim  h^{\frac \sigma p}  \lambda^{-\frac{1}{dp}} $. Hence \eqref{tlaw} implies
\[
h^{\frac \sigma p} \lambda^{-\frac\alpha q- \frac{1}{dp}}\gtrsim
h^\sigma \lambda^{-\frac{1}d}  \big( \mu(\mathcal R)\big)^\frac1q.
\]
By a computation $ \mu(\mathcal R)\sim
\lambda^{-\alpha+\frac{\beta(\alpha)}d}.$ Hence,
$\lambda^{-\frac\alpha q- \frac{1}{dp}} \gtrsim \lambda^{-\frac{1}d}
\lambda^{-\frac\alpha q+\frac{\beta(\alpha)}{dq} }$. Letting
$\lambda\to \infty$ gives the condition \eqref{necc}.

\section{Proof of Lemma \ref{1-1}.  }
\label{appendb}

Here we provide a proof of Lemma \ref{1-1}. For $1\le n\le d$\, let
us set
\[E_n = \{ \mbt \in I^n : 0< t_1 < \dots <t_n\}.\] We need  to show that $\Gamma_\gamma : E_d \to \Real^d$ is one-to-one
provided that $\gamma \in \mathfrak G^\mba (\epsilon)$ and
$\epsilon>0$ is sufficiently small. Since $\gamma\in \mathfrak
G^\mba (\epsilon)$, it is obvious that the determinant of
$\frac{\partial \Gamma_\gamma(\mbt)}{\partial \mbt}$ and  all its
minors  take the form $\det(q'(t_{\alpha_1}),\dots,
q'(t_{\alpha_n}))$ while $\alpha_1,\dots, \alpha_n\in \{1,\dots,
d\}$ and $q\in \mathfrak G^\mbb(\epsilon)$ for some  $\mbb =
(b_1,\dots,b_n)$, $b_1<\ldots<b_n$, $b_1,\dots,b_n\in \{a_1,\dots,
a_d\}$. Here $\mathfrak G^\mbb(\epsilon)$ and $\gamma^\mbb_\circ$ is
defined similarly  as before by replacing $n$ for $d$. Hence,
by the argument in \cite{DM2} (also see \cite[Section 6]{DeW}) which
is originally due to Steinig \cite{steinig}, we only need to show
that $\det(q'(t_{\alpha_1}),\dots, q'(t_{\alpha_n}))$ is single
signed and nonzero for  $(t_{\alpha_1},\dots, t_{\alpha_n})\in E_n$
provided that $q\in \mathfrak G^\mbb(\epsilon)$ and  $\epsilon>0$ is
sufficiently small. Therefore the following lemma completes the
proof.

\begin{lemma}\label{jalow} For $1\le n\le d$,  let\, $\mbb =
(b_1,\cdots,b_n)$ and $b_1, b_2, \dots, b_n$ be positive integers
satisfying that $b_1<b_2< \ldots <b_n$.  Let $\gamma\in \mathfrak
G^\mbb(\epsilon)$ and set $\Gamma_\gamma(\mbt)=\sum_{i=1}^n
\gamma(t_i)$, $\mbt=(t_1,\dots, t_n)\in I^n$.  Then if
$\epsilon=\epsilon\,(\mbb, n)>0$ is sufficiently small, there is a
constant $C$, independent of $\gamma$, such that if
$\mbt=(t_1,\dots, t_n)\in E_n$,
\begin{align}\label{finitejacobian}
&\det \bigg(\frac{\partial \Gamma_\gamma(\mbt)}{\partial \mbt}\bigg)
 \ge C \prod_{i=1}^n \Big| \det \Big( (\gamma^\mbb_\circ)'(t_i)
,\cdots,(\gamma^\mbb_\circ)^{(n)}(t_i) \Big)\Big|^{\frac{1}{n}}
\prod_{1\leq i<j\leq n}(t_j - t_i).
\end{align}
\end{lemma}

\begin{proof} We shall be
brief since the proof here is an adaptation of the argument in
\cite{DeM}. Let $\Phi_k(t)$ be a $k\times k $ minor of $\Phi_n(t) :=
\det \Phi(t)$, which consists of $\Phi_{i,j}(t)$ with $1\leq i, j
\leq k$. (See Lemma \ref{finitetorsion}.)

Adopting the notations in \cite{DeW, DeM}, we define a sequence of
functions  $I_k$, $1\leq k \leq n$ as follows:
\begin{align*}
I_1 (t) = \frac{ t^{\sum_{i=1}^{n-2} (a_i - i)} \Phi_{n-2}(t) \,
t^{\sum_{i=1}^n (a_i - i)} \Phi_n(t) }{(t^{\sum_{i=1}^{n-1} (a_i-i)}
\Phi_{n-1}(t) )^2}  & = t^{a_n- a_{n-1} -1} \frac{\Phi_{n-2}(t)
\Phi_n(t) }{\Phi_{n-1}(t)^2}
\end{align*}
and
\begin{align*}
 I_k(t_1,\dots,t_k) & = \prod_{l=1}^k t_l^{a_{n-k+1} - a_{n-k}-1}
\frac{\Phi_{n-k-1}(t_l)\Phi_{n-k+1}(t_l)}{\Phi_{n-k}(t_l)^2}\\
&\qquad\qquad\times\int_{t_1}^{t_2} \cdots\int_{t_{k-1}}^{t_k}
I_{k-1}(s_1, \dots, s_{k-1}) d s_{k-1} \cdots d s_1
\end{align*}
with $\Phi_{-1}, \Phi_0 \equiv 1$.
By Lemma \ref{finitetorsion} and Remark \ref{positivity}, there are
positive constants $G_k$, uniform for
$\gamma\in\mathfrak{G}^\mbb(\epsilon)$, such that $\frac{1}{2} G_k
\leq \Phi_k (t) \leq 2 G_k$ for all $t \in I$ and sufficiently small
$\varepsilon>0$. Hence $I_1(t) \gtrsim t^{a_n-a_{n-1}-1}\, G_{n-2}
\, G_n /G_{n-1}^2 $.

Now we claim that for $1\leq k \leq n-2$,
\begin{align}\label{intseq}
I_k (t_i , \dots,t_k)\gtrsim \frac{G_{n-k-1}^k \,
G_n}{G_{n-k}^{k+1}} \prod_{l=1}^k t_l^{\frac{1}{k} \sum_{i=n-k+1}^n
(a_i -i) - (a_{n-k} -(n-k))} \prod_{1\leq i< j\leq k}(t_j - t_i)
\end{align}
also holds uniformly in $\gamma \in \mathfrak{G}^\mbb(\epsilon)$.
Suppose that \eqref{intseq} holds for $k \leq n-3$. Then, it follows
that
\begin{align*}
&\quad I_{k+1} (t_1,\dots,t_{k+1} ) \gtrsim  \left( \frac{G_{n-k-2}
G_{n-k}}{G_{n-k-1}^2} \right)^{k+1} \frac{G_{n-k-1}^k \,
G_n}{G_{n-k}^{k+1}} \prod_{l=1}^{k+1} t_l^{a_{n-k} -a_{n-k-1} -1}
\\
& \quad\qquad \times \int_{t_1}^{t_2} \cdots \int_{t_k}^{t_{k+1}}
\prod_{l=1}^k s_l^{(\frac{1}{k} \sum_{i=n-k+1}^n (a_i -i) - (a_{n-k} -(n-k)))}
 \prod_{1 \leq i < j \leq k}(t_j - t_i) d s_k \cdots d s_1
 \\
&  \gtrsim \frac{G_{n-k-2}^{k+1} G_n}{G_{n-k-1}^{k+2}}
\prod_{l=1}^{k+1} t_l^{(\frac{1}{k+1} \sum_{i=n-k}^n (a_i -i) -
(a_{n-k-1} -(n-k-1)))} \prod_{1\leq i< j\leq k+1}(t_j - t_i).
\end{align*}
The first inequality is valid uniformly for $\gamma \in
\mathfrak{G}^\mbb(\varepsilon_0)$, and the last inequality is
established by modifying Corollary 7 in \cite{DeM}. 
The remaining cases $k=n-1,
n$ can also be handled similarly by making use of $\eqref{intseq}$
successively. So, it follows that
\begin{align*}
I_n(t_1,\dots,t_n) \gtrsim G_n \prod_{l=1}^n t_l^{\frac{1}{n}
\sum_{i=1}^n(a_i-i)} \prod_{1\leq i < j\leq n}(t_j-t_i)
\end{align*}
holds uniformly. Since $I_n (t_1, \dots,t_n) =  \det \partial
\Gamma_\gamma(\mbt)/
\partial \mbt  $ (see Section 5 in \cite{DeW}) and
$t_l^{\sum_{i=1}^n (a_i -i)} \sim |\det ( (\gamma^\mbb_\circ)'(t_l),
\dots,(\gamma^\mbb_\circ)^{(n)}(t_l) )|$, we conclude that
\eqref{finitejacobian} holds uniformly  for
$\gamma\in\mathfrak{G}^\mbb(\varepsilon_0)$.
\end{proof}



\begin{thebibliography}{9}

\bibitem{act}  G.I. Arkhipov,  V.N. Chubarikov and A.A.  Karatsuba,
 Trigonometric sums in number theory and analysis, Translated
from the 1987 Russian original. de Gruyter Expositions in
Mathematics, 39, Berlin, 2004.

\bibitem{becata} J. Bennett, A. Carbery and T. Tao,
On the multilinear restriction and Kakeya conjectures,
\textit{Acta. Math.}, {\bf 196} (2006), 261--302.

\bibitem{bl} J.-G. Bak and S. Lee,
Estimates for an oscillatory integral operator related to
restriction to space curves, \textit{Proc. Amer. Math. Soc.},
\textbf{132} (2004), 1393--1401.


\bibitem{bs} J. Bennett and A. Seeger,
The Fourier extension operator on large spheres and related
oscillatory integrals, \textit{Proc. London. Math. Soc.},
\textbf{98} (2009), 45--82.


\bibitem{BO} J.-G. Bak and D. Oberlin,  {A note on Fourier restriction for
curves in $\mathbb R^3$,}  Proceedings of the AMS Conference on
Harmonic Analysis,
 Mt. Holyoke College (June 2001), Contemp. Math., Vol. 320,
Amer. Math. Soc., Providence, RI, 2003.

\bibitem{BOS2} J.-G. Bak, D. Oberlin and  A. Seeger,
Restriction of Fourier transforms to curves, II: Some classes with
vanishing torsion, \textit{J. Austr. Math. Soc.}, \textbf{85}
(2008), 1--28.


\bibitem{BOS1} \bysame,
Restriction of Fourier transforms to curves and related oscillatory
integrals, \textit{Amer. J. Math.}, \textbf{131} (2009), 277--311.

\bibitem{BOS3} \bysame,
Restriction of Fourier transforms to curves: An endpoint estimate
with affine arclength measure, \textit{J. Reine Angew. Math.}, \textbf{682} (2013), 167--206.


\bibitem{bcsv}  J. Bennett, A. Carbery, F. Soria and A. Vargas,
A Stein conjecture for the circle, \textit{Math. Annalen.},
\textbf{336} (2006), 671--695.


\bibitem{bg} J. Bourgain and L. Guth,
Bounds on oscillatory integral operators based on multilinear
estimates, \textit{Geom. Funct. Anal.}, \textbf{21} (2011),
1239--1295.



\bibitem{christ} M. Christ,
On the restriction of the Fourier transform to curves: endpoint
results and the degenerate case, \textit{Trans. Amer. Math. Soc.},
\textbf{287} (1985), 223--238.


\bibitem{DeFW} S.  Dendrinos, M.  Folch-Gabayet and J.  Wright,
{ An affine-invariant inequality for rational functions and
applications in harmonic analysis,} \textit{Proc. Edinb. Math.
Soc.}, \textbf{53} (2010), 639--655.

\bibitem{DLW} S. Dendrinos, N.  Laghi and J.  Wright,
{Universal $L^p$ improving for averages along polynomial curves in
low dimensions,} \textit{J. Funct. Anal.}, \textbf{257}  (2009),
1355--1378.


\bibitem{DeM} S. Dendrinos and D. M\"uller,
{Uniform estimates for the local restriction of the Fourier
transform to curves}, \textit{Trans. Amer. Math. Soc.}, \textbf{365} (2013), 3477--3492.

\bibitem{DeW} S. Dendrinos and J. Wright,
{Fourier restriction to polynomial curves I: A geometric
inequality,} \textit{Amer. J. Math.}, \textbf{132} (2010),
1031--1076.


\bibitem{drury} S. W. Drury,
Restriction of Fourier transforms to curves, \textit{Ann. Inst.
Fourier.}, \textbf{35} (1985), 117--123.


\bibitem{D2} \bysame,
{Degenerate  curves and harmonic analysis,} \textit{Math. Proc.
Cambridge Philos. Soc.}, \textbf{108} (1990), 89--96.

\bibitem {DM1} S.W.  Drury and B. Marshall,
{Fourier restriction theorems for curves with affine and Euclidean
arclengths,} \textit{Math. Proc. Cambridge Philos. Soc.}, \textbf{97}
(1985), 111--125.

\bibitem {DM2}\bysame,
{Fourier restriction theorems for degenerate curves,} \textit{Math.
Proc. Cambridge Philos. Soc.}, \textbf{101} (1987), 541--553.

\bibitem {F} C. Fefferman,
Inequalities for strongly singular convolution operators,
\textit{Acta. Math.}, \textbf{124} (1970), 9--36.


\bibitem{Fol} G. Folland,
Real Analysis; Modern Techniques and their Applications,
Wiley-Interscience, New York, 1999.

\bibitem{grs} A. Greenleaf and A. Seeger,
On oscillatory integral operators with folding canonical relations,
\textit{Studia Math.}, \textbf{132} (1999), 125--139.


\bibitem {H}L. H\"ormander, {Oscillatory integrals and multipliers on
$FL^{p}$,} \textit{Ark. Mat.}, \textbf{11} (1973), 1--11.




\bibitem{ob} D. Oberlin,
{Fourier restriction estimates for affine arclength measures in the
plane,} \textit{Proc. Amer. Math. Soc.}, \textbf{129} (2001), 3303--3305.

\bibitem{presti} E. Prestini, Restriction theorems for the Fourier transform to come manifolds in $\Real^n$, \textit{Proc. Sympos. Pure. Math.}, \textbf{35} (1979), 101--109.

\bibitem{sj} P. Sj\"olin,
{Fourier multipliers and estimates of the Fourier transform of
measures carried by smooth curves in $R\sp{2}$,} \textit{Studia Math.},
\textbf{51}  (1974), 169--182.


\bibitem{steinig} J. Steinig, On some rules of Laguerre's, and systems of equal sums of like
powers, \textit{Rend. Mat.}, (6) 4 (1971), 629--644 (1972).

\bibitem{tavave} T. Tao, A. Vargas and L. Vega, {A bilinear
approach to the restriction and Kakeya conjectures,} \textit{J.
Amer. Math. Soc.}, {\bf 11} (1998), 967--1000.


\bibitem{Z} A.Zygmund,
{On Fourier coefficients and transforms of functions of two
variables}, \textit{Studia Math.}, \textbf{50} (1974), 189--201.
\end{thebibliography}
\end{document}